\documentclass[onecolumn, draftcls,12pt]{IEEEtran}

\usepackage{amssymb,amsmath,amsfonts,amsthm}
\usepackage{algorithm,algorithmic}
\usepackage{boxedminipage}
\usepackage{cite}
\usepackage[dvips]{graphicx}
\usepackage{epsfig}
\usepackage{float}
\graphicspath{{figures/}}
\usepackage{graphicx}
\usepackage{graphics}
\usepackage{multirow}
\usepackage{threeparttable} 
\usepackage[usenames,dvipsnames]{color}
\usepackage{url,bm,xspace,dsfont}

\usepackage{verbatim}

\def\diag{\mathop{\mathrm{diag}}}  

\newtheorem{theorem}{Theorem}
\newtheorem{lemma}{Lemma}
\newtheorem{definition}{Definition}
\newtheorem{corollary}{Corollary}

\newtheorem{remark}{Remark}

\newtheorem{assumptions}{Assumptions}
\newtheorem{problem}{Problem}

\newcommand{\eps}{\epsilon}

\newcommand{\la}{\langle}
\newcommand{\ra}{\rangle}

\newcommand{\sr}{\stackrel}

\newcommand{\rar}{\rightarrow}

\newcommand{\tri}{\sr{\triangle}{=}}

\newcommand{\be}{\begin{equation}}
\newcommand{\ee}{\end{equation}}
\newcommand{\bea}{\begin{eqnarray}}
\newcommand{\eea}{\end{eqnarray}}
\newcommand{\bes}{\begin{eqnarray*}}
\newcommand{\ees}{\end{eqnarray*}}

\newcommand{\bi}{\begin{itemize}}
\newcommand{\ei}{\end{itemize}}
\newcommand{\ben}{\begin{enumerate}}
\newcommand{\een}{\end{enumerate}}

\newcommand{\bp}{\begin{problem}}
\newcommand{\ep}{\end{problem}}
\newcommand{\hso}{\hspace{.1in}}
\newcommand{\hst}{\hspace{.2in}}

\newcommand{\noi}{\noindent}

\newcommand{\bc}{\begin{center}}
\newcommand{\ec}{\end{center}}

\floatstyle{ruled}
\newfloat{algorithm}{H}{algo}
\floatname{algorithm}{\footnotesize Algorithm}

\begin{document}
%
\title{\bf Centralized Versus Decentralized Team Games of  Distributed Stochastic Differential Decision Systems with Noiseless Information Structures-Part II: Applications}


\author{ Charalambos D. Charalambous\thanks{C.D. Charalambous is with the Department of Electrical and Computer Engineering, University of Cyprus, Nicosia 1678 (E-mail: chadcha@ucy.ac.cy).}  \: and Nasir U. Ahmed\thanks{N.U Ahmed  is with the School of  Engineering and Computer Science, and Department of Mathematics,
University of Ottawa, Ontario, Canada, K1N 6N5 (E-mail:  ahmed@site.uottawa.ca).}
}

\maketitle

\begin{abstract}
In this second part of our two-part paper, we invoke the stochastic maximum principle, conditional Hamiltonian and the coupled backward-forward stochastic differential equations of the first part \cite{charalambous-ahmedFIS_Parti2012} to derive team optimal decentralized strategies for   distributed stochastic differential systems with noiseless information structures. We present examples of such team games of nonlinear  as well as linear quadratic forms. In some cases we obtain closed form expressions of the optimal  decentralized strategies.  

Through the examples, we illustrate the effect of information signaling  among the decision makers in reducing the computational complexity of optimal decentralized decision strategies. 
\end{abstract}

\vspace*{1.5cm}

  
  \vspace*{1.5cm}

  \section{Introduction}
\label{introduction}
In the first part \cite{charalambous-ahmedFIS_Parti2012} of this two part paper, we have derived   team and person-by-person optimality conditions for distributed stochastic differential systems with noiseless decentralized information structures. Specifically, we considered distributed (coupled) stochastic differential equations of It\^o form driven  by Brownian motions, and decision makers acting on decentralized noiseless i) nonanticipative and ii) feedback information structures, and   we have  shown existence of team and person-by-person optimal strategies utilizing relaxed and regular strategies. Then we applied tools from the classical theory of stochastic optimization with some variations to derive team and person-by-person optimality conditions \cite{bismut1978,ahmed1981,ahmed-charalambous2012a,charalambous-hibey1996}.

The first important concussions drawn from \cite{charalambous-ahmedFIS_Parti2012} is that the classical theory of stochastic optimization  is not limited in mathematical concepts and procedures by the centralized assumption based upon which it is developed. It is directly applicable to differential  systems  consisting of multiple decision makers, in which  the acquisition of information and its processing is decentralized or shared among several locations, while the   decision makers  actions are based on different information structures. The second important conclusion drawn from \cite{charalambous-ahmedFIS_Parti2012} is that team and person-by-person optimality conditions are given by a Hamiltonian system of equations consisting of a conditional Hamiltonian, and  coupled forward-backward stochastic differential equations. 

 The work  in \cite{charalambous-ahmedFIS_Parti2012}  compliments the current body of knowledge on  static team game theory  \cite{marschak1955,radner1962,marschak-radner1972,krainak-speyer-marcus1982a,waal-vanschuppen2000}, and decentralized decision making \cite{witsenhausen1968,witsenhausen1971,ho-chu1972,kurtaran-sivan1973,sandell-athans1974,kurtaran1975,varaiya-walrand1978,ho1980,bagghi-basar1980,krainak-speyer-marcus1982a,krainak-speyer-marcus1982b,waal-vanschuppen2000}, and more recent work in \cite{bamieh-voulgaris2005,aicardi-davoli-minciardi1987,nayyar-mahajan-teneketzis2011,vanschuppen2011,lessard-lall2011,mahajan-martins-rotkowitz-yuksel2012},  by introducing optimility conditions for general stochastic nonlinear differential systems.  

 The main remaining challenge is to determine whether under the formulation and assumptions introduced in \cite{charalambous-ahmedFIS_Parti2012}, we can derive optimal decentralized strategies for nonlinear and linear distributed stochastic differential systems,  understand the computational complexity of these strategies compared to centralized strategies,  and determine how this complexity can be reduced by allowing limited signaling among the different decision makers. 

 Therefore, in this second part of the two-part investigation, we apply the optimality conditions derived in the first part to a variety of linear and nonlinear distributed stochastic differential systems with decentralized noiseless information structures to derive optimal strategies. Our  investigation leads to the following conclusions.

\ben
\item When the dynamics are linear in the decision variables and  nonlinear in the state variables, and the pay-off is quadratic in the decision variable and nonlinear in the state variable, the optimal decentralized strategies are given in terms of conditional expectations with respect to  the information structure on which they act on;  

\item When the dynamics are linear in the state and the decision variables, and the pay-off is quadratic in the state and the decision variables, then the optimal decentralized strategies are computed in closed form, much as in the classical Linear-Quadratic Theory. However, when the pay-off includes coupling between the decision makes the optimal strategy of any player is also a function of the average value of the optimal strategies of the other players. 

\item The computation of the optimal strategies involves the solution of certain equations, which can be formulated and solved via fixed point methods. 

\item The computation complexity of the optimal  decentralized strategies can be reduced by  signaling specific information among the decision makers and/or by considering certain structure for the distributed system and pay-off.
\een

The rest of the paper is organized as follows. In Section~\ref{formulation},  we introduce  the distributed stochastic system with decentralized information structures and the main assumption, and we state  the optimality conditions derived in  \cite{charalambous-ahmedFIS_Parti2012}. In Section~\ref{classes},  we apply to optimality conditions to several forms of team games, and we show how the optimal decentralized strategies are computed. For the case of linear differential dynamics and quadratic pay-off we obtain explicit expressions of the optimal decentralized team strategies.   The  paper is concluded with some comments on possible extensions of our results.

 \section{Team and Person-by-Person Optimality Conditions}
 \label{formulation}

\noi In this section we introduce  the mathematical formulation of distributed stochastic  systems with decentralized noiseless  information structures, and the optimality conditions derived in \cite{charalambous-ahmedFIS_Parti2012}. 

\noi The formulation in \cite{charalambous-ahmedFIS_Parti2012}  presupposes a fixed probability space with filtration,   $\Big(\Omega,{\mathbb F},  \{ {\mathbb F}_{0,t}:   t \in [0, T]\}, {\mathbb P}\Big)$  satisfying the usual conditions, that is,  $(\Omega,{\mathbb F}, {\mathbb P})$ is complete, ${\mathbb F}_{0,0}$ contains all ${\mathbb P}$-null sets in ${\mathbb F}$. All $\sigma-$algebras are assumed complete and right continuous, that is,   ${\mathbb F}_{0,t} = {\mathbb F}_{0,t+} \tri \bigcap_{s>t} {\mathbb F}_{0,s}, \forall t \in [0,T)$. We use the notation  ${\mathbb F}_T \tri  \{ {\mathbb F}_{0,t}:   t \in [0, T]\}$ and similarly for the rest of the filtrations. \\  

\noi The minimum principle in  \cite{charalambous-ahmedFIS_Parti2012} is derived utilizing the  following spaces. Let $L_{{\mathbb F}_T}^2([0,T],{\mathbb R}^n) \subset   L^2( \Omega \times [0,T], d{\mathbb P}\times dt,  {\mathbb R}^n) \equiv L^2([0,T], L^2(\Omega, {\mathbb R}^n)) $ denote the space of ${\mathbb F}_T-$adapted random processes $\{z(t): t \in [0,T]\}$   such that
\bes
{\mathbb  E}\int_{[0,T]} |z(t)|_{{\mathbb R}^n}^2 dt < \infty,
\ees   
 which is a sub-Hilbert space of     $L^2([0,T], L^2(\Omega, {\mathbb R}^n))$.      
  Similarly, let $L_{{\mathbb F}_T}^2([0,T],  {\cal L}({\mathbb R}^m,{\mathbb R}^n)) \subset L^2([0,T] , L^2(\Omega, {\cal L}({\mathbb R}^m,{\mathbb R}^n)))$ denote the space of  ${\mathbb  F}_{T}-$adapted $n\times m$ matrix valued random processes $\{ \Sigma(t): t \in [0,T]\}$ such that
 \bes
  {\mathbb  E}\int_{[0,T]} |\Sigma(t)|_{{\cal L}({\mathbb R}^m,{\mathbb R}^n)}^2 dt  \tri  {\mathbb E} \int_{[0,T]} tr(\Sigma^*(t)\Sigma(t)) dt < \infty.
  \ees

  \subsection{Distributed Stochastic Differential Decision Systems }
\label{deterministic}

A stochastic differential decision or control system is called distributed if it consists of an interconnection of at least two subsystems and decision makers, whose   actions are based on decentralized information structures.  The underlying assumption is that  the decision makers are allowed to  exchange  information on their  law or strategy deployed,  but not their actions. \\
  Let $\Big(\Omega,{\mathbb F},\{ {\mathbb F}_{0,t}:   t \in [0, T]\}, {\mathbb P}\Big)$ denote a fixed complete filtered probability space on which we shall define all processes. At this state we do not specify how $\{{\mathbb F}_{0,t}: t \in [0,T]\}$ came about, but we require that Brownian motions are adapted to this filtration.\\
   
  \noi {\bf Admissible Decision Maker  Strategies}
  
  \noi  The Decision Makers (DM) $\{u^i: i \in {\mathbb Z}_N\}$ take values in a closed convex subset of metric spaces $\{({\mathbb M}^i,d): i \in {\mathbb Z}_N\}$.  Let  ${\cal G}_{T}^i \tri \{  {\cal G}_{0,t}^i: t \in [0, T]\} \subset \{{\mathbb F}_{0,t}: t \in [0,T]\}$ denote the information available to DM $i$, $ \forall i \in {\mathbb Z}_N$.   
     The admissible set  of regular strategies is  defined by
 \begin{align}
 {\mathbb U}_{reg}^i[0, T] \tri \Big\{   u^i   \in  L_{{\cal G}_T^i}^2([0,T],{\mathbb R}^{d_i})  : \:   u_t^i \in {\mathbb A}^i \subset {\mathbb R}^{d_i}, \: a.e. t \in [0,T], \: {\mathbb P}-a.s. \Big\}, \hso  \forall  i \in {\mathbb Z}_N. \label{cs1a}
     \end{align}
 Clearly,   ${\mathbb U}_{reg}^i[0, T]$ is a  closed convex subset of  $L_{{\mathbb F}_T}^2([0,T],{\mathbb R}^n)$, for $i=1,2, \ldots, N$, and    $u^i  : [0,T] \times \Omega \rar {\mathbb A}^i$,  
   $\{u_t^i: t  \in  [0, T] \}$  
    is  ${\cal G}_T^i-$adapted, $\forall  i \in {\mathbb Z}_N$.\\
     An $N$ tuple of DM strategies    is by definition $(u^1,u^2, \ldots, u^N) \in {\mathbb U}_{reg}^{(N)}[0,T] \tri \times_{i=1}^N {\mathbb U}_{reg}^i[0,T]$.\\


 \noi {\bf Distributed Stochastic  Systems}

  \noi  On the probability space $\Big(\Omega,{\mathbb F},\{ {\mathbb F}_{0,t}:   t \in [0, T]\}, {\mathbb P}\Big)$ the distributed stochastic system consists of an interconnection of $N$ subsystems, and  each subsystem $i$ has,  state space ${\mathbb R}^{n_i}$, action space ${\mathbb A}^i \subset {\mathbb R}^{d_i}$, an exogenous   noise space $ {\mathbb R}^{m_i}$, and an  initial state $x^i(0)=x_0^i$, identified by the following quantities.
\begin{description}
\item[(S1)] $x^i(0)=x_0^i$:  an ${\mathbb R}^{n_i}$-valued  Random Variable;

\item[(S2)] $\{ W^i(t): t \in [0,T]\}$: an  ${\mathbb R}^{m_i}$-valued  standard Brownian motion which models the exogenous state noise, adapted to ${\mathbb F}_T$, independent of $x^i(0)$;
\end{description}
Each subsystem is described by  coupled stochastic differential equations of It\^o type as follows. 
 \begin{align}
 dx^i(t) =&  f^i(t,x^i(t),u_t^i) dt  +\sigma^i(t,x^i(t),u_t^i)dW^i(t)
 + \sum_{j=1, j \neq i}^N f^{ij}(t,x^j(t),u_t^j)dt  \nonumber \\
 &+\sum_{j=1, j \neq i}^N \sigma^{ij}(t,x^j(t),u_t^j)dW^j(t) , \hst x^i(0) = x_0^i, \hso  t \in (0,T], \hso i \in {\mathbb Z}_N. \label{eq1ds}
 \end{align}  
Define the augmented vectors by
\bes
 W \tri (W^1, W^2, \ldots, W^{N}) \in {\mathbb R}^m, \hso u\tri (u^1, u^2, \ldots, u^{N}) \in {\mathbb R}^d, \hso x \tri (x^1, x^2, \ldots, x^{N}) \in {\mathbb R}^n.
 \ees
  \\  
The distributed   system  is described in compact form by
 \begin{eqnarray}
 dx(t) =  f(t,x(t),u_t) dt + \sigma(t,x(t),u_t)~dW(t), \hst x(0) = x_0, \hst  t \in (0,T], \label{eq1}
 \end{eqnarray}  
 where $f: [0,T] \times {\mathbb R}^n\times {\mathbb A}^{(N)} \longrightarrow {\mathbb R}^n$ denotes the drift and $\sigma : [0,T] \times {\mathbb R}^n\times {\mathbb  A}^{(N)} \longrightarrow {\cal L}({\mathbb R}^m, {\mathbb R}^n)$ the diffusion coefficients. \\
 
   \noi {\bf Pay-off Functional}
   
  \noi Given a $u \in {\mathbb U}_{reg}^{(N)}[0,T]$ and (\ref{eq1ds}) we  define the reward or performance criterion by 
 \begin{align} 
 J(u) \equiv  J(u^1,u^2,\ldots,u^N)   \tri
    {\mathbb E} \biggl\{   \int_{0}^{T}  \ell(t,x(t),u_t) dt + \varphi(x(T)\biggr\},            \label{cfd}
  \end{align} 
  where $\ell: [0,T] \times {\mathbb R}^n\times {\mathbb U}^{(N)} \longrightarrow (-\infty, \infty]$ denotes the running cost function  and $\varphi : {\mathbb R}^n \longrightarrow (-\infty, \infty]$,  the terminal cost function.  

\subsection{Team and Person-by-Person Optimality}
\label{tpbp}
In this section we give the precise definitions of team and person-by-person optimality for regular strategies.\\

We consider the following information structures.\\

\noi{ \bf (NIS): Nonanticipative Information Structures.}   $u^i$ is adapted to the filtration
  ${\cal G}_T^i \subset {\mathbb F}_T$  generated by the $\sigma-$algbebra of nonlinear nonanticipative measurable functionals of any combination of the subsystems Brownian motions $\{(W^1(t),W^2(t), \ldots, W^N(t)): t \in [0,T]\}, \forall i \in {\mathbb Z}_N$. \\
  This is often called open loop information, and  it is the one used in classical stochastic control with centralized full information  to derive the maximum principe \cite{yong-zhou1999}. \\

\noi{ \bf (FIS): Feedback Information Structures.} $u^i$ is adapted   
to the filtration ${\cal G}_T^{z^{i,u}}$ generated by the $\sigma-$algebra   ${\cal G}_{0,t}^{z^{i,u}} \tri  \sigma\{ z^{i}(s): 0 \leq s \leq t\}, t \in [0,T]$, where  the observables $z^i$ are nonlinear nonanticipative measurable functionals of any combination of the states defined  by
\bea
z^i(t) = h^i(t,x), \hst h^i : [0,T] \times C([0,T], {\mathbb R}^n)  \longrightarrow {\mathbb R}^{k_i}, \hso i \in {\mathbb Z}_N. \label{eq1aa}
\eea
Note that the index $u$ emphasizes the fact that feedback strategies depend on $u$. \\
The set of admissible regular  feedback strategies is defined by 
\bea
{\mathbb U}_{reg}^{(N),z^u}[0,T] \tri \Big\{u \in  {\mathbb U}_{reg}^{(N)}[0,T]: u_t^i \hso \mbox{is} \hso {\cal G}_{0,t}^{z^{i,u}}-\mbox{measurable},  t \in [0,T], \hso i=1,\ldots, N\Big\}. \label{fstrategies1}
\eea

 \begin{problem}(Team Optimality)
 \label{problem1}
Given the  pay-off functional (\ref{cfd}),   constraint (\ref{eq1})   the  $N$ tuple of  strategies   $u^o \tri (u^{1,o}, u^{2,o}, \ldots, u^{N,o}) \in {\mathbb U}_{reg}^{(N)}[0,T]$  is called nonanticipative team optimal if it satisfies  
 \bea
 J(u^{1,o}, u^{2,o}, \ldots, u^{N,o}) \leq J(u^1, u^2, \ldots, u^N),  \hst \forall  u\tri (u^1, u^2, \ldots, u^N) \in {\mathbb U}_{reg}^{(N)}[0,T] \label{cfd1a}
 \eea
 Any $u^o   \in {\mathbb U}_{reg}^{(N)}[0,T]$ satisfying (\ref{cfd1a})
is called an optimal   decision strategy (or control) and the corresponding $x^o(\cdot)\equiv x(\cdot; u^o(\cdot))$ (satisfying (\ref{cfd})) is called an optimal state process.
Similarly,  feedback 
team optimal strategies  are defined with respect to  $u^o  \in {\mathbb U}_{reg}^{(N),z^u}[0,T]$. 
  \end{problem}
 
  An alternative approach to handle such problems with decentralized information structures is to restrict the definition of optimality to   the so-called person-by-person equilibrium.  \\
 Define 
 \bes
 \tilde{J}(v,u^{-i})\tri J(u^1,u^2,\ldots, u^{i-1},v,u^{i+1},\ldots,u^N)
 \ees

 \begin{problem}(Person-by-Person Optimality)
 \label{problem2}
 Given the   pay-off functional (\ref{cfd}),   constraint (\ref{eq1})   the  $N$ tuple of  strategies   $u^o \tri (u^{1,o}, u^{2,o}, \ldots, u^{N,o}) \in {\mathbb U}_{reg}^{(N)}[0,T]$  is called nonanticipative person-by-person optimal  if it satisfies
\begin{align}
 \tilde{J}(u^{i,o}, u^{-i,o}) \leq \tilde{J}(u^{i}, u^{-i,o}), \hst \forall u^i \in {\mathbb  U}_{reg}^i[0,T], \hso \forall i \in {\mathbb Z}_N. \label{cfd2}
 \end{align}
 Similarly,  feedback 
 person-by-person optimal strategies are defined with respect to  $u^o \in {\mathbb U}_{reg}^{(N),z^u}[0,T]$. 
  \end{problem}

 \noi   Conditions (\ref{cfd2}) are analogous to the Nash equilibrium strategies of team games consisting of a single pay-off and $N$ DM. The person-by-person optimal strategy states that none of the $N$ DM with different  information structures can deviate unilaterally from the optimal strategy and gain by doing so.

   \subsection{Team and Person-by-Person Optimality Conditions}
  \label{mp}

In this section we first introduce the assumptions on $\{f,\sigma,h,\ell,\varphi\}$ and then we state the optimality conditions derived in \cite{charalambous-ahmedFIS_Parti2012}. 
   
\noi Let $B_{{\mathbb F}_T}^{\infty}([0,T], L^2(\Omega,{\mathbb R}^n))$ denote the space of ${\mathbb F}_T$-adapted ${\mathbb R}^n$ valued second order random processes endowed with the norm topology  $\parallel  \cdot \parallel$ defined by  
\bes
 \parallel x\parallel^2  \tri \sup_{t \in [0,T]}  {\mathbb E}|x(t)|_{{\mathbb R}^n}^2.
 \ees
The main assumptions are stated below.

\begin{assumptions}(Main assumptions)\\
\label{NCD1}  
${\mathbb U}^i$ is closed and convex subset of ${\mathbb R}^{d_i}$, $\forall  i\in {\mathbb Z}_N$, ${\mathbb E}|x(0)|_{{\mathbb R}^n} < \infty$ and the maps of $\{f,\sigma,\ell, \varphi\} $ satisfy the following conditions.

\begin{description}

\item[(A1)] $f: [0,T] \times {\mathbb R}^n \times {\mathbb A}^{(N)} \longrightarrow {\mathbb R}^n$ is continuous in $(t,x,u)$ and continously differentiable with respect to $x,u$;

\item[(A2)] $\sigma: [0,T] \times {\mathbb R}^n \times {\mathbb A}^{(N)} \longrightarrow {\cal L}({\mathbb R}^m; {\mathbb R}^n)$ is continuous in $(t,x,u)$ and continously differentiable with respect to $x,u$;

\item[(A3)] The first derivatives of $\{f_x, \sigma_x, f_u, \sigma_u \}$ are bounded  uniformly on $[0,T] \times {\mathbb R}^n \times {\mathbb A}^{(N)}$.

\item[(A4)] $\ell: [0,T] \times {\mathbb R}^n \times {\mathbb A}^{(N)} \longrightarrow (-\infty, \infty]$ is Borel measurable, continuously differentiable with respect to  $(x,u)$,  $\varphi: [0,T] \times {\mathbb R}^n \longrightarrow (-\infty, \infty]$ is continously differentiable with respect to $x$, $\ell(0,0,t)$ is bounded, and there exist $K_1, K_2 >0$ such that
\bes
|\ell_x(t,x,u)|_{{\mathbb R}^n}+|\ell_u(t,x,u) |_{{\mathbb R}^d}       \leq K_1 \big(1+|x|_{{\mathbb R}^n} + |u|_{{\mathbb R}^d} \big),  \hso |\varphi_x(x)|_{{\mathbb R}^n} \leq K_2 \big(1+ |x|_{{\mathbb R}^n}\big).
\ees
\end{description}

\end{assumptions}

\noi The following lemma states  existence of solutions and their continuous dependence on the decision variables.

\begin{lemma}
\label{lemma3.1}
 Suppose Assumptions~\ref{NCD1} hold. Then for any ${\mathbb  F}_{0,0}$-measurable initial state $x_0$ having finite second moment, and any $u \in {\mathbb U}_{reg}^{(N)}[0,T]$,  the following hold.

\begin{description}

 \item[(1)]  System (\ref{eq1}) has a unique solution   $x \in B_{{\mathbb F}_T}^{\infty}([0,T],L^2(\Omega,{\mathbb R}^n))$  having a continuous modification, that is, $x \in C([0,T],{\mathbb R}^n)$,  ${\mathbb P}-$a.s, $\forall i \in {\mathbb Z}_N$.

\item[(2)]  The solution of  system  (\ref{eq1}) is continuously dependent on the control, in the sense that, as $u^{i, \alpha} \longrightarrow u^{i,o}$  in ${\mathbb U}_{reg}^i[0,T]$, $\forall i \in {\mathbb Z}_N$,  $x^\alpha \buildrel s \over\longrightarrow x^o $ in $B_{{\mathbb F}_T}^{\infty}([0,T],L^2(\Omega,{\mathbb R}^n)), \forall i \in {\mathbb Z}_N$.
\end{description}
These statements also hold for feedback  strategies $u \in {\mathbb U}_{reg}^{(N),z^u}[0,T]$. 

\end{lemma}

\begin{proof} 	Proof is identical to that of  \cite{ahmed-charalambous2012a}.
\end{proof}

\noi Note that the differentiability of $f, \sigma, \ell$ with respect to $u$ can be removed without affecting the results (by considering either needle variations when deriving the maximum principle or by deriving the maximum principle for relaxed strategies and then specializing it to regular strategies as  in \cite{charalambous-ahmedFIS_Parti2012}).

\noi Assumptions~\ref{NCD1}  are used to  derive optimality conditions for stochastic control problems with  nonanticipative centralized  strategies. However, for stochastic control problems with feedback centralized strategies additional assumptions are required to avoid certain technicalities associated with the derivation of the maximum principle. In \cite{charalambous-ahmedFIS_Parti2012} we identified these assumptions for decentralized randomized feedback strategies;  the main theorems  are stated below.

\begin{assumptions} 
\label{a-nf}
The following holds.

\begin{description}
\item[(E1)] The diffusion coefficients $\sigma$  is restricted to   the   map $\sigma: [0,T] \times {\mathbb R}^n \longrightarrow {\cal L}({\mathbb R}^n, {\mathbb R}^n)$ 
(e.g., it is independent of $u$) and $\sigma(\cdot,\cdot)$ and $\sigma^{-1}(\cdot, \cdot)$  are bounded.
\end{description}
\end{assumptions} 

\noi Define the $\sigma-$algebras 
\bes
{\cal F}_{0,t}^{x(0),W} \tri \sigma\{x(0), W(s): 0\leq s \leq t\}, \hst   {\cal F}_{0,t}^{x^u} \tri \sigma\{x(s): 0 \leq s \leq t\}, \hst \forall t \in [0,T].
\ees
Under Assumptions~\ref{NCD1}, \ref{a-nf}, if  $u \in {\mathbb U}_{reg}^{(N),z^{u}}[0,T]$  then ${\cal F}_{0,t}^{x(0),W} = {\cal F}_{0,t}^{x^u}, \forall t \in [0,T]$. Thus,   for any $u^i \in {\mathbb U}_{reg}^{z^{i,u}}[0,T]$ which is ${\cal G}_T^{z^{i,u}}-$adapted there exists a function $\phi^i(\cdot)$ measurable to a sub-$\sigma-$algebra of ${\cal F}_{0,t}^i \subset {\cal F}_{0,t}^{x(0),W}$ such that $u_t^i( \omega)= \phi^i(t, x(0), W(\cdot \bigwedge t,\omega)), {\mathbb P}-a.s. \: \omega \in {\Omega}, \forall t \in [0,T], i=1,\ldots N$. 
Define all such adapted nonanticipative functions by 
\begin{align}
\overline{\mathbb U}_{reg}^{i}[0, T] \tri \Big\{   u^i   \in  L_{{\cal F}_{T}^{ i}}^2([0,T],{\mathbb R}^{d_i})  : \:   u_t^i \in  {\mathbb U}_{reg}^{i,z^{i,u}}[0, T]     \Big\}, \: \forall  i \in {\mathbb Z}_N.\label{na1}
     \end{align}

\noi Next, we introduce the following additional assumptions.

\begin{assumptions}
\label{a-nf1r}
The following hold. 
\begin{description}
\item[(E2)] ${\mathbb U}_{reg}^{z^{i,u}}[0, T]$ is dense in  $\overline{\mathbb U}_{reg}^{i}[0, T], \forall i \in {\mathbb Z}_N$.

\end{description}
\end{assumptions} 

\noi Under Assumptions~\ref{NCD1} it can be shown that $J(\cdot)$ is continuous in the sense of $\overline{\mathbb U}_{reg}^{(N)}[0,T]$ and by Assumptions~\ref{a-nf1r} we have  $\inf_{ u \in  \times_{i=1}^N \overline{\mathbb U}_{reg}^{i}[0,T]} J(u)=  \inf_{ u \in \times_{i=1}^N {\mathbb U}_{reg}^{z^{i,u}}[0,T]} J(u)$. Hence,  the  necessary conditions for feedback information structures $u \in {\mathbb U}_{reg}^{(N),z^{u}}[0,T]$ to be optimal are those for which nonanticipative information structures $u \in \overline{\mathbb U}_{reg}^{(N)}[0,T]$ are optimal.

\noi We now show  that under Assumptions~\ref{NCD1}, \ref{a-nf} then   Assumptions~\ref{a-nf1r} holds.

\begin{theorem}
\label{thm-nf1}
Consider Problem~\ref{problem1} under Assumptions~\ref{NCD1}, \ref{a-nf}.  
Then   $$\inf_{ u \in  \times_{i=1}^N \overline{\mathbb U}_{reg}^{i}[0,T]} J(u)=  \inf_{ u \in \times_{i=1}^N {\mathbb U}_{reg}^{z^{i,u}}[0,T]} J(u).$$
\end{theorem}

\begin{proof} We follow the procedure in \cite{bensoussan1981}.  For any $u^i \in {\mathbb U}_{reg}^{z^{i,u}}[0,T]$ which is ${\cal G}_T^{z^{i,u}}-$adapted we can define the set $\overline{\mathbb U}_{reg}^i [0,T], i=1, \ldots, N$ via (\ref{na1}). Let $u \in \overline{\mathbb U}_{reg}^{(N)}[0,T] \tri \times_{i=1}^N \overline{\mathbb U}_{reg}^i[0,T]$ and for $k =\frac{T}{M}$, define
\bea
u_{k,t}^i = \left\{ \begin{array}{cccc} u_0^i & \mbox{for} & 0 \leq t <k & u_0 \in {\mathbb A}^{i} \\
\frac{1}{k} \int_{(n-1)k}^{nk} u_s^i ds & \mbox{for} & nk \leq t  (n+1)k, & n=1,\ldots, M-1, \end{array} \right.
\eea
for $i=1, \ldots N$.
Then $u_{k,t} \equiv (u_{k,t}^1, \ldots, u_{k,t}^N) \in \overline{\mathbb U}_{reg}^{(N)}[0,T] $, and $u_k \longrightarrow u$ in $
L_{{\cal F}_T}^2([0,T], {\mathbb R}^d )$. We need to show that $u_k \in {\mathbb U}^{(N),z^{u_k}}[0,T]$. Let $x_k(\cdot)$ denote the trajectory corresponding to $u_{k,\cdot}$, and ${\cal F}_{0,t}^{x_k^u}$ the $\sigma-$algebra generated by $\{x_k(s): 0 \leq s \leq t\}$. Define
\begin{align}
I_k(t) \tri \int_{0}^t \sigma(s,x_k(s))  dW(t)  
=& x_k(t) -x(0) -\int_{0}^t f(s, x_k(s), u_{k,s}) ds, \label{dc1}
\end{align}
and
\bea
W(t) = \int_{0}^t \sigma(s,x_k(s))^{-1} d I_k(s). \label{dc2}
\eea
Since $u_k \in \overline{\mathbb U}_{reg}^{(N)}[0,T] $, then $I_k(t)$ is ${\cal F}_{0,t}^{x_k^u}-$measurable, for $0 \leq t <k$. Hence,
\bea
{\cal F}_{0,t}^{x(0),W}= {\cal F}_{0,t}^{x_k^u}, \hst 0\leq t \leq k. \label{dc3}
\eea
Therefore, $u_{k,t}$ is $ {\cal F}_{0,t}^{x_k^u}-$measurable for $k \leq t \leq 2k$.  From the above equations it follows that (\ref{dc3}) also holds for $k \leq t \leq 2k$, and by induction that ${\cal F}_{0,t}^{x(0),W} = {\cal F}_{0,t}^{x_k^u}, \forall t \in [0,T]$. Therefore,  $u_{k,t}^i $ is also measurable with respect to ${\cal F}_{0,t}^{x_k^u}$. Hence, for any $u_t^i$  which is measurable with respect to a nonanticipative functional $z^i=h^i(t,x)$ there exists a nonanticipative functional of $x(0), W$ which realizes it. \\
Now, it is sufficient to show that  as $u^{i, \alpha} \longrightarrow u^{i}$  in $\overline{\mathbb U}_{reg}^i[0,T]$, $\forall i \in {\mathbb Z}_N$,  then $J(u^\alpha) \longrightarrow J(u)$. Utilizing Assumptions~\ref{NCD1} we can  show that ${\mathbb E} \sup_{s \in [0,t] } |x^\alpha(s)-x(s)|_{{\mathbb R}^n}$ converges to zero as $\alpha \longrightarrow 0$, hence it is sufficient to show that $|J(u^\alpha)-J(u)|$ also converges to zero, as $\alpha \longrightarrow 0$. By the mean value theorem we have the following inequality.
\begin{align}
| J(u^\alpha)-J(u)| \leq & K_1 \: {\mathbb E} \Big\{ \int_{[0,T]}^{} \Big( |x^\alpha(t)|_{{\mathbb R}^n} + |u_t^\alpha|_{{\mathbb R}^d} +  |x(t)|_{{\mathbb R}^n} + |u_t|_{{\mathbb R}^d} +1 \Big) \nonumber \\
&.\Big(  |x^\alpha(t)-x(t)|_{{\mathbb R}^n} + |u_t^\alpha-u_t|_{{\mathbb R}^d} \Big)dt\Big\} \nonumber \\
&+ K_2 {\mathbb E} \Big\{ \Big( |x^\alpha(T)|_{{\mathbb R}^n} + |x(T)|_{{\mathbb R}^n} +1 \Big) |x^\alpha(T)-x(t)|_{{\mathbb R}^n} \Big\}. \label{nac}
\end{align}
Since ${\mathbb E} \sup_{s \in [0,t] } |x^\alpha(s)-x(s)|_{{\mathbb R}^n} \longrightarrow 0$ as $\alpha \longrightarrow 0$, then $|J(u^\alpha)-J(u)|$ also converges to zero, as $\alpha \longrightarrow 0$.

\end{proof}

\noi Thus, under the assumptions of  Theorem~\ref{thm-nf1}  if $u \in {\mathbb U}_{reg}^{(N),z^{u}}[0,T]$ achieves the infimum of $J(u)$ then it is also optimal with respect to  $\overline{\mathbb U}_{reg}^{(N)}[0,T] \tri \times_{i=1}^N \overline{\mathbb U}_{rel}^{i}[0,T] $. Consequently, the  necessary conditions for feedback information structures $u \in {\mathbb U}_{reg}^{(N),z^{u}}[0,T]$ to be optimal are those for which nonanticipative information structures $u \in \overline{\mathbb U}_{reg}^{(N)}[0,T]$ are optimal.


\noi In the next remark  we give an example for which Assumptions~\ref{a-nf} hold, and hence Theorem~\ref{thm-nf1} is valid.

\begin{remark}
\label{ss}
Suppose $x^1$ and $x^2$ are governed by the following stochastic differential equations
\begin{align}
dx^1(t)=& f^1(t,x^1(t),u^1(t))dt + \sigma^1(t,x^1(t))dW^1(t), \hst x^1(0)=x_0^1, \label{ss1} \\
dx^2(t)=& f^2(t,x^1(t),x^2(t),u^1(t),u^2(t))dt + \sigma^2(t,x^1(t),x^2(t))dW^2(t), \hst x^2(0)=x_0^2, \label{ss2} \\
z^1(t)=&h^1(t,x^1(t)), \hst z^2(t)=h^2(t,x^1(t),x^2(t)), \hst t \in [0,T], \label{ss3}
\end{align}
where $h^1, h^2$ are measurable, $W^1(\cdot), W^2(\cdot)$ are independent, and $u^1 \in {\mathbb U}_{reg}^{1,z^{1,u}}[0,T], u^2 \in {\mathbb U}_{reg}^{2,z^{1,u}, z^{2,u}}[0,T]$. If we further assume that $\sigma^i(\cdot, \cdot), \sigma^{i,-1}(\cdot,\cdot)$ are bounded, and Assumptions~\ref{NCD1} hold, then ${\cal F}_{0,t}^{x^{1,u}}\tri \sigma\{x^1(s): 0\leq s \leq t\}={\cal F}_{0,t}^{x^1(0),W^1} \tri \sigma\{x^1(0), W(s): 0 \leq s \leq t\}$. Moreover, it can be shown that  ${\cal F}_{0,t}^{x^{1,u}, x^{2,u}} = {\cal F}_{0,t}^{x^1(0), x^2(0), W^1, W^2}$. Then we can find $\overline{\mathbb U}_{reg}^i[0,T], i=1,2$ for which { (E2)} holds, and thus  Theorem~\ref{thm-nf1} is valid. 
\end{remark}

\noi Next, we state the main theorem which gives  necessary and sufficient optimality conditions  for nonanticipative and feedback  decisions.\\
\noi Define the Hamiltonian
\bes
 {\cal  H}: [0, T] \times {\mathbb R}^n\times {\mathbb R}^n\times {\cal L}({\mathbb R}^m,{\mathbb R}^n)\times {\mathbb A}^{(N)} \longrightarrow {\mathbb R},
\ees  
   by  
   \begin{align}
    {\cal H} (t,x,\psi,Q,u) \tri \la f(t,x,u),\psi\ra + tr (Q^*\sigma(t,x,u))
     + \ell(t,x,u),  \hst  t \in  [0, T]. \label{dh1}
    \end{align}
  \noi For any $u \in {\mathbb U}_{reg}^{(N)}[0,T]$, the adjoint process is $(\psi,Q) \in  L_{{\mathbb F}_T}^2([0,T], {\mathbb R}^n)\times L_{{\mathbb F}_T}^2([0,T] ,{\cal L}({\mathbb R}^m,{\mathbb R}^n))$ satisfies the following backward stochastic differential equation
\begin{align} 
d\psi (t)  &= -f_x^{*}(t,x(t),u_t)\psi (t)  dt - V_{Q}(t) dt -\ell_x(t,x(t),u_t) dt + Q(t) dW(t), \hst t \in [0,T),    \nonumber    \\ 
&=- {\cal H}_x (t,x(t),\psi(t),Q(t),u_t) dt + Q(t)dW(t),      \label{adj1a} \\ 
 \psi(T) &=   \varphi_x(x(T))  \label{eq18}  
 \end{align}
  where  $V_{Q} \in L_{{\mathbb F}_T}^2([0,T],{\mathbb R}^n)$ is   given by  $\langle V_{Q}(t),\zeta\rangle = tr (Q^*(t)\sigma_x(t,x(t),u_t; \zeta)), t \in [0,T]$ (e.g., $V_{Q}(t) = \sum_{k=1}^m \Big(\sigma_x^{(k)}(t,x(t),u_t)\Big)^*Q^{(k)}(t), \hst t \in [0,T],$ $\sigma^{(k)}$ is the $kth$ column of $\sigma$, $\sigma_x^{(k)}$ is the  derivative of $\sigma^{(k)}$ with respect to the state, for $k=1, 2, \ldots, m$, $Q^{(k)}$ is the $kth$ column of $Q$).\\ 
The state process satisfies the stochastic differential equation
\begin{align}
dx(t) &=f(t,x(t),u_t)dt + \sigma(t,x(t),u_t)dW(t), \hst t \in (0, T],  \nonumber  \\
        & = {\cal H}_\psi (t,x(t),\psi(t),Q(t),u_t)     dt + \sigma(t,x(t),u_t) dW(t), \label{st1a} \\
        x(0) &=  x_0 \label{st1i}  
 \end{align}

\noi The main theorem is stated below.

\begin{theorem}(Team optimality) 
\label{theorem7.1}
 Consider Problem~\ref{problem1} under Assumptions~\ref{NCD1}, and  assume  existence of an optimal team strategy.
 
\begin{description}

\item[(I)] Suppose ${\mathbb F}_T$ is the filtration generated by $x(0)$ and the Brownian motion $\{W(t): t \in [0,T]\}$.
  
\noi {\bf Necessary Conditions.}    For  an element $ u^o \in {\mathbb U}_{reg}^{(N)}[0,T]$ with the corresponding solution $x^o \in B_{{\mathbb F}_T}^{\infty}([0,T], L^2(\Omega,{\mathbb R}^n))$ to be team optimal, it is necessary  that 
the following hold.

\begin{description}

\item[(1)]  There exists a semi martingale  with the intensity process $({\psi}^o,Q^o) \in  L_{{\mathbb F}_T}^2([0,T],{\mathbb R}^n)\times L_{{\mathbb F}_T}^2([0,T],{\cal L}({\mathbb R}^m,{\mathbb R}^n))$.
 
 \item[(2) ]  The variational inequality is satisfied:

\begin{align}     \sum_{i=1}^N {\mathbb  E} \Big\{ \int_0^T  \la  {\cal H}_{u^i} (t,x^o(t),\psi^o(t), Q^{o}(t), u_t^{o}),u_t^i-u_t^{i,o}\ra dt \Big\}\geq 0, \hst \forall u \in {\mathbb U}_{reg}^{(N)}[0,T]. \label{eqd16}
\end{align}

\item[(3)]  The process $({\psi}^o,Q^o) \in  L_{{\mathbb F}_T}^2([0,T],{\mathbb R}^n)\times L_{{\mathbb F}_T}^2([0,T],{\cal L}({\mathbb R}^m,{\mathbb R}^n))$ is a unique solution of the backward stochastic differential equation (\ref{adj1a}), (\ref{eq18}), such that $u^o \in {\mathbb U}_{reg}^{(N)}[0,T]$ satisfies  the  point wise almost sure inequalities with respect to the $\sigma$-algebras ${\cal G}_{0,t}^i   \subset {\mathbb F}_{0,t}$, $ t\in [0, T], i=1, 2, \ldots, N:$ 

\begin{align} 
\la  {\mathbb E} \Big\{   {\cal H}_{u^i}(t,x^o(t),  &\psi^o(t),Q^o(t),u_t^{o})   |{\cal G}_{0, t}^i \Big\} ,  u^i-u_t^{i,o}\ra \geq  0, \nonumber \\
&  \forall u^i \in {\mathbb A}^i,  a.e. t \in [0,T], {\mathbb P}|_{{\cal G}_{0,t}^i}- a.s., i=1,2,\ldots, N   \label{eqhd35} 
\end{align} 

\noi {\bf Sufficient Conditions.}    Let $(x^o(\cdot), u^o(\cdot))$ denote an admissible state and decision pair and let $\psi^o(\cdot)$ the corresponding adjoint processes. \\
   Suppose the following conditions hold.
   
\begin{description}

\item[(B1)]  ${\cal H}(t, \cdot,\psi,Q, \cdot),   t \in  [0, T]$,  is convex in $(x,u) \in {\mathbb R}^n \times {\mathbb A}^{(N)}$;

\item[(B2)]  $\varphi(\cdot)$,  is convex in $x \in {\mathbb R}^n$.

\end{description}

\end{description}

\noi Then $(x^o(\cdot),u^o(\cdot))$ is optimal if it satisfies (\ref{eqhd35}).
   
\item[(II)] Suppose ${\mathbb F}_T$ is the filtration generated by $x(0)$ and the Brownian motion $\{W(t): t \in [0,T]\}$, and Assumptions~\ref{a-nf1r} hold. The necessary and sufficient conditions for a feedback  element $ u^o \in {\mathbb U}_{reg}^{(N),z^u}[0,T]$ to be optimal are given by the statements under Part {\bf (I)} with ${\cal G}_{0,t}^i$ replaced by ${\cal G}_{0,t}^{z^{i,u}}, \forall t \in [0,T]$. 

\end{description}   
   
\end{theorem}

\begin{proof} See \cite{charalambous-ahmedFIS_Partii2012}.

\end{proof}

\noi Next, we have the following corollary regarding  person-by-person optimality.

\begin{corollary}(Person-by-person optimality)
\label{corollarypbpd}
Consider Problem~\ref{problem2} under the conditions of Theorem~\ref{theorem7.1}. Then the necessary and sufficient condition of Theorem~\ref{theorem7.1} hold with variational inequality (\ref{eqd16}) replaced by
\begin{align}     {\mathbb  E} \Big\{ \int_0^T  \la  {\cal H}_{u^i} (t,x^o(t),\psi^o(t), Q^{o}(t), u_t^{o}),u_t^i-u_t^{i,o}\ra dt \Big\}\geq 0, \hst \forall u^i \in {\mathbb U}_{reg}^{i}[0,T], \hso \forall i \in {\mathbb Z}_N. \label{eqd16pp}
\end{align}
\end{corollary}

\begin{proof}   See \cite{charalambous-ahmedFIS_Parti2012}.
\end{proof}

\noi It can be shown by contradiction that the team and person-by-person  optimality conditions presented above are equivalent 
 (see \cite{charalambous-ahmedFIS_Parti2012}).\\

\noi Often in the application of the minimum principle we need to identify the the martingale term in the adjoint process equation. One approach how to determine $Q$ is discussed in the next remark.

\begin{remark}
\label{mart}
Utilizing the Riesz representation theorem for Hilbert space martingles,  in \cite{charalambous-ahmedFIS_Parti2012} the adjoint process $Q(\cdot)$  in the adjoint equation (\ref{adj1a}), is identified as $ Q(t) \equiv \psi_x(t) \sigma(t,x(t),u_t)$, provided $\psi_x$ exists (i.e., $f, \sigma, \ell, \varphi$ are twice continuously differentiable and $f_{xx}, \sigma_{xx}, \ell_{xx}, \varphi_{xx}$  are uniformly bounded).
\end{remark}

\noi Note that from the team optimality conditions presented above we also deduce the optimality conditions for centralized full and partial information strategies. This observation is stated in the next remark (for partial information strategies)

\begin{remark}
\label{remark5.2} 
 Consider Problem~\ref{problem1}  under the conditions of Theorem~\ref{theorem7.1}, Part {\bf (I)} and Part {\bf (II)} and assume $u^i$ are adapted  the centralized partial information ${\cal G}_T\subset {\mathbb F}_T$, and centralized partial information  ${\cal G}_T^{z^u} \subset {\cal F}_{0,T}^{x^u}$, respectively. Then the necessary conditions for ${\cal G}_T-$adapted $u^i$'s are given by the following point wise almost sure inequalities 

\begin{eqnarray} 
 {\mathbb E} \Big\{  {\cal H}(t,x^o(t),\psi^o(t),Q^o(t),u )|{\cal G}_{0,t}\Big\} \geq   {\mathbb E}\Big\{  {\cal H}(t,x^o(t),\psi(t),Q(t),u_t^o)|{\cal G}_{0,t}\Big\}, \nonumber \\
 \forall u \in {\mathbb A}^{(N)}, a.e. t \in [0,T], {\mathbb P}|_{{\cal G}_{0,t}}-a.s.,  \label{eq35} 
 \end{eqnarray} 
 where $\{x^o(t), \psi^o(t), Q^o(t): t \in [0,T]\}$ are the solutions of the Hamiltonian system (\ref{st1a}), (\ref{st1i}), (\ref{adj1a}), (\ref{eq18}), while for  ${\cal G}_T^{z^u}-$adapted $u^i$'s  the necessary condition is (\ref{eq35}) with the conditioning done with respect to ${\cal G}_T^{z^u}$.  This corresponds to the partial information investigated in \cite{ahmed-charalambous2012a}.

\end{remark}

\section{Optimal Team Strategies for Classes of Games}
\label{classes}
We are now ready to derive  explicit optimal team strategies for general classes of team games,  when the  dynamics and the reward have certain structures. These  include nonlinear as well as linear distributed systems. Our focus is on optimal decentralized strategies which are given in a) closed form involving conditional expectations based on the information structures available to the DM's and b) closed expressions similar to the classical Linear-Quadratic Theory. \\
First, we define the main classes of team games we shall investigate.

\begin{definition}(Team games with special structures)
\label{normal}
We define the following forms of team games.

\noi{\bf (GNF): Generalized Normal Form.} The team game is said to have "generalized normal form" if 
\begin{align}
f(t,x,u) \tri &b(t,x)+ g(t,x)u, \hst g(t,x)u \tri \sum_{j=1}^N g^{(j)}(t,x) u^j,  \label{gn1} \\
 \sigma(t,x,u)  \tri & \left[ \begin{array}{cccc}  \kappa^{(1)}(t,x) & \kappa^{(2)}(t,x)  \ldots & \kappa^{(m)}(t,x)  \end{array} \right]  \nonumber \\
 &+ 
  \left[ \begin{array}{cccc}  s_{1}(t,x)u & s_{2}(t,x)u & \ldots & s_{m}(t,x)u\end{array} \right], \label{gn2} \\
 \ell(t,x,u) \tri &\frac{1}{2}\la u,R(t,x)u\ra + \frac{1}{2} \lambda(t,x)+ \la u, \eta(t,x)\ra , \label{gn3}  \\
\mbox{where} \hso \la u, R(t,x) u \ra \tri & \sum_{i=1}^N \sum_{j=1}^N u^{i,*} R_{ij}(t,x)u^j, \hst \la u,\eta(t,x)\ra \tri  \sum_{i=1}^N u^{i,*}\eta^i(t,x), \nonumber 
\end{align}
and   $\kappa^{(i)}(\cdot,\cdot)$ is the $i$th column of an $n\times m$ matrix $\kappa(\cdot,\cdot)$, for $i=1,2, \ldots, m$, $s_i(\cdot,\cdot)$ is an $n\times d$ matrix,  for $i=1,2, \ldots, m$, $R(\cdot,\cdot)$ is symmetric uniformly positive definite, and $\lambda(\cdot,\cdot)$ is uniformly positive semidefinite.\\
{\bf GNF} refers to the case when the drift and diffusion coefficients $f,\sigma$ are linear in the decision variable $u$, and the pay-off  function $\ell$ is quadratic in $u$, while $f,\sigma, \ell, \varphi$ are nonlinear in $x$.\\ 

\noi{\bf (SGNF): Simplified Generalized Normal Form.} A team game is  said  to have "simplified generalized normal form" if it is of generalized normal form and  $\sigma(t,x,u) $ is independent of $u$, that is, $s_j=0, 1 \leq j \leq m$ in (\ref{gn2}).   \\
{\bf SGNF} refers to the case when $f$ is linear in $u$, $\sigma$ is independent of $u$, $\ell $ is quadratic in $u$, and $f, \sigma, \ell$ are nonlinear in $x$.\\

\noi{\bf (NF): Normal Form.} A team game is  said  to have "normal form" if     
\begin{align}
f(t,x,u)=&A(t)x+b(t) +  B(t)u,  \label{n1} \\
\sigma(t,x,u)  \tri & \left[ \begin{array}{cccc}  \kappa_1(t) x & \kappa_2(t)x & \ldots & \kappa_m(t)x  \end{array} \right] \nonumber \\
& + 
  \left[ \begin{array}{cccc}  s_{1}(t)u & s_{2}(t)u & \ldots & s_{m}(t)u\end{array} \right]+G(t), \label{n1a}\\
 \ell(t,x)=&\frac{1}{2}\la u,R(t)u\ra  + \frac{1}{2} \la x,H(t)x\ra +\la x,F(t)\ra + \la u, E(t)x\ra  +\la u,m(t)\ra ,  \label{n3}\\
  \varphi(x) =& \frac{1}{2} \la x, M(T)x \ra + \la x, N(T)\ra ,  \label{n3a}
\end{align}
and   $\kappa_i(\cdot) \in {\cal L}({\mathbb R}^n; {\mathbb R}^n)$ for $i=1,\ldots, m$,   $s_i(\cdot) \in {\cal L}({\mathbb R}^d; {\mathbb R}^n)$   for $i=1, \ldots, m$, and $R(\cdot)$ is symmetric uniformly positive definite, $H(\cdot)$ is symmetric uniformly positive semidefinite, and $M(T)$ is symmetric positive  semidefinite.\\
{\bf NF} refers to the case when $f, \sigma$ are linear is $x, u$, and $\ell, \varphi$ are quadratic in $x, u$. Therefore, the dynamics also include stochastic integral terms which are linear is $x, u$.\\

\noi{\bf (LQF): Linear-Quadratic Form.} A team game is  said  to have "normal form" if     
\begin{align}
f(t,x,u)=&A(t)x+  B(t)u,  \hst \sigma(t,x,u)= G(t), \label{lq1} \\
 \ell(t,x)=&\frac{1}{2}\la u,R(t)u\ra  + \frac{1}{2} \la x,H(t)x\ra, \hst
 \varphi(x) =\frac{1}{2}\la x, M(T) x\ra , \label{lg3}
\end{align}
and $R(\cdot)$ is symmetric uniformly positive definite, $H(\cdot)$ is symmetric uniformly positive semidefinite, and $M(T)$ is symmetric positive semidefinite.\\
{\bf NF} refers to the case when $f$ is linear in $x, u$, $\sigma$ is independent of $x, u$, and $\ell, \varphi$ are quadratic in $x,u$; this is the classical linear-quadratic (dynamics, pay-off) model often utilized in centralized decision making.  

\end{definition}

\noi Below we compute the optimal strategies for the different cases of Definition~\ref{normal}. Although, we utilized nonanticipative strategies ${\mathbb U}_{reg}^{(N)}[0,T]$, these computations can be done for feedback strategies ${\mathbb U}_{reg}^{(N),z^u}[0,T]$.\\

\noi {\bf  Case {\bf GNF}}. \\
Utilizing the definition of Hamiltonian (\ref{dh1}), its derivative is given by
\bea
 {\cal H}_u(t,x,\psi,Q,u)  = g^*(t,x) \psi  +  \sum_{i=1}^m  s_i^*(t,x) Q^{(i)}(t) +  R(t,x)u + \eta(t,x), \hso   (t,x)\in [0,T]\times {\mathbb R}^n. \label{dh3}
 \eea
  Since the diffusion coefficient $\sigma$ depends on  $u$, then  $Q^o(\cdot)$  also depends on the control and by Remark~\ref{mart},  $Q^o$ is given by 
\begin{align}
Q^o(t) =&\psi_x^o(t) \sigma(t,x^o(t),u_t^o), \label{mt1a}  \\
 Q^{(i),o}(t) =& \psi_x^o(t)\Big\{  \kappa^{(i)}(t,x^o) + s_i(t,x^o)u_t^o\Big\}, \hst t \in [0,T], \hso   i=1,2, \ldots, m. \label{mt1}
\end{align}
 Define the quantities 
\begin{align}
\Lambda(t,x, \psi_x) \tri & \sum_{i=1}^m s_i^*(t,x) \psi_x(t) \kappa^{(i)}(t,x) \equiv \left[ \begin{array}{c} \Lambda^1 \\ \Lambda^2 \\ \ldots \\ \Lambda^N \end{array} \right](t,x,\psi_x), \hst \nonumber \\
M(t,x,\psi_x) \tri & \sum_{i=1}^m s_i^*(t,x) \psi_x(t) s_i(t,x) \equiv \left[ \begin{array}{cccc} M_{11} & M_{12} & \ldots & M_{1N} \\
M_{21} & M_{22} & \ldots& M_{2N} \\
\ldots & \ldots & \ldots &\ldots \\ 
M_{N1} & M_{N2} & \ldots & M_{NN} \end{array} \right](t,x,\psi_x), \nonumber 
\end{align}
 where  $\Lambda^i \in {\cal L}({\mathbb R}^{m}, {\mathbb R}^{d_i})$,  $M_{ij} \in {\cal L}({\mathbb R}^{d_j}, {\mathbb R}^{d_i})$,  $i,j=1,2, \ldots, N$. \\
  By Theorem~\ref{theorem7.1},  substituting    $Q^o(\cdot)$ given by  (\ref{mt1}) into (\ref{dh3}), and utilizing the fact that $u_t^{i,o}$ is ${\cal G}_{0,t}^i-$adapted for each $i \in {\mathbb Z}_N$, the explicit expression for $u_t^{i,o}$ is obtained from (\ref{eqhd35}), and it is given by
 \begin{align}
u_t^{i,o}=&- \Big\{  {\mathbb E} \Big( (R_{ii} +M_{ii})(t,x^o(t),\psi_x^o(t))  | {\cal G}_{0,t}^i\Big)\Big\}^{-1} \Big\{ {\mathbb E} \Big( \eta^i(t,x^o(t)) + \Lambda^i(t,x^o(t), \psi_x^o(t)) | {\cal G}_{0,t}^i\Big)  \nonumber \\
& + \sum_{j=1,j \neq i}^N  {\mathbb E} \Big( (R_{ij}+M_{ij})(t,x^o(t),\psi_x^o(t))u_t^{j,o}     | {\cal G}_{0,t}^i \Big)    \nonumber     \\
  & +    {\mathbb  E} \Big(  g^{(i),*}(t,x)\psi^{o}(t) |{\cal G}_{0,t}^i\Big) \Big\}, \hst  {\mathbb P}|_{ {\cal G}_{0,t}^i}-a.s., \hso i=1,2, \ldots, N .\label{grn4}
\end{align}
Next, we make some observations. \\

\begin{description}
\item[(O1):]
At any $t \in [0,T]$, $u_t^{i,o}$ is a functional of estimates of all other optimal decisions $u_t^{j,o}, j\neq i$ given its own information. Such strategies impose a  heavy computational burden on any  decentralized decision maker. Therefore, a question which might be of interest to address, is "what information needs to be signal among the DM's  to reduce computations?" The answer to this question will become apparent when we proceed to compute the explicit expressions of the optimal strategies.

\item[(O2):]
In the simplified case of diagonal $M, R$,  the right side of (\ref{grn4}) does not depend directly on estimates of the other DM's, but  this dependence is hidden in the adjoint process $\psi^o(\cdot)$. In fact, since no communication exchange is allowed between the DM's, then  any communication between them is made via the interaction of the DM's with the state and adjoint processes $x^o(\cdot), \psi^o(\cdot)$. One may view the stochastic differential system together with the adjoint backward stochastic differential equation as playing the role of a channel that makes communication between the DM's possible. Therefore, an interesting question is "can  we quantify the amount of information communicated among the DM's via the Hamiltonian system of equations and if so, can we utilize this insight to reduce the computational burden, by allowing limited signaling between the DM's?" We shall return to this question and identify the variable which are involved in  such communication between the DMs.
\end{description}

Finally, note that the optimal strategies can be  further simplified by assuming $g(t,x)$ is linear in $x$ and $\sigma(t,\cdot, u)$ is linear in $x$, $R(\cdot, \cdot)$ is independent of $x$, $\lambda(\cdot,\cdot)$ is quadratic in $x$, and $\eta(\cdot,\cdot)$ is linear in $x$. \\

\noi {\bf  Case SNF.} \\
\noi For a team game of simplified  generalized  form, the diffusion coefficient $\sigma$ is independent of $u$,  therefore the second right hand side term of (\ref{dh3}) is zero (since $s_i=0, i=1,2, \ldots, m$), and the derivative of the Hamiltonian is linear in $u$. Therefore, the explicit expressions  for $u^{i,o}$ are obtained  from (\ref{grn4}) by setting $M_{ij}=0, \Lambda^i=0, i,j=1, \ldots, N$, hence 
 \begin{align}
u_t^{i,o}=&- \Big\{  {\mathbb E} \Big( R_{ii}(t,x^o(t))  | {\cal G}_{0,t}^i\Big)\Big\}^{-1} \Big\{ {\mathbb E} \Big( \eta^i(t,x^o(t))  | {\cal G}_{0,t}^i\Big)  
 + \sum_{j=1,j \neq i}^N  {\mathbb E} \Big( R_{ij}(t,x^o(t))u_t^{j,o}     | {\cal G}_{0,t}^i \Big)    \nonumber     \\
  & -     {\mathbb  E} \Big(  g^{(i),*}(t,x)\psi^{o}(t) |{\cal G}_{0,t}^i\Big) \Big\}, \hst  {\mathbb P}|_{ {\cal G}_{0,t}^i}-a.s., \hso i=1,2, \ldots, N .\label{sgrn4}
\end{align}
By comparing the optimal strategies for  {\bf GNF} given by (\ref{grn4}) and (\ref{sgrn4}) we have the following observation.

\begin{description}
\item[(O3):]
 When $\sigma$ is independent of $u$,  then the adjoint process $Q^o(\cdot)$ is independent of $u^o$, and therefore the optimal strategies do not involve derivatives of the adjoint process $\psi_x(\cdot)$ as in (\ref{grn4}). 
 
 \item[(O4):] The team game formulation also includes as a special case, distributed estimation as follows. Suppose each component of the vector $x \tri (x^1, \ldots x^N)$ denotes the channel output at  different  distributed receivers carrying an information message (RV) $\theta \tri (\theta^1, \ldots \theta^N), \theta^i: \Omega \longrightarrow {\mathbb R}^{n_i}, i=1, \ldots, N$, and each channel is  subject to feedback and interference from the other channels. Then each channel outputs can be described by   
\begin{align}
dx^i(t)=& b^i(t,x^i(t),\theta^i)dt +\kappa^i(t,x^i(t))dW^i + \sum_{j=1, j \neq i}^N b^{ij}(t,x^j(t),\theta^j)dt \nonumber \\
& +\sum_{j=1, j \neq i}^N \kappa^{ij}(t,x^j(t))dW^j, \hst x^i(0)=x_0^i, \hso i=1, \ldots N .\label{af2}
\end{align}
 Thus,  $\{x^i(t): 0\leq t \leq T\}$ describes  the channel output of the $i$th receiver which is  subject to feedback and interference from the other channels,  $\theta^i$ is the message to be estimated at the $ith$ receiver, and $u_t^i(\{x^i(s): 0\leq s \leq t\})$ is its team optimal estimator at time $t \in [0,T]$,  based on  having access to $x^i$. Then the optimal distributed team estimators are obtained from  
     (\ref{sgrn4}), and they are given  by the following equation. 
  \begin{align}
u_t^{i,o}=- \Big\{  {\mathbb E} \Big( R_{ii}(t,x(t))  | {\cal G}_{0,t}^{x^i}\Big)\Big\}^{-1} \Big\{ {\mathbb E}& \Big( \eta^i(t,x(t))  | {\cal G}_{0,t}^{x^i}\Big)  
 + \sum_{j=1,j \neq i}^N  {\mathbb E} \Big( R_{ij}(t,x(t))u_t^{j,o}     | {\cal G}_{0,t}^{x^i} \Big)  \Big\},  \nonumber     \\
  & \hst  {\mathbb P}|_{ {\cal G}_{0,t}^{x^i}}-a.s., \hso i=1,2, \ldots, N .\label{af1}
\end{align}   
One may consider several other scenarios of distributed estimation by considering specific pay-off function $\ell(t,x,u)$ which represents estimation error.

\end{description}

\noi {\bf Case NF.} \\
For a team game of normal form define the quantities 
\begin{align}
\overline{\Lambda}(t,x,\psi_x) \tri & \sum_{i=1}^m s_i^*(t) \psi_x(t) \Big(\kappa_i(t)x+G^{(i)}(t)\Big) \equiv \left[ \begin{array}{c}  \overline{\Lambda^1}(t,x,\psi_x)  \\ \overline{\Lambda^2}(t,x,\psi_x)   \\ \ldots \\ \overline{\Lambda^N}(t,x,\psi_x) \end{array} \right], \hst \overline{\Lambda^i} \in {\cal L}({\mathbb R}^{n}, {\mathbb R}^{d_i}), i=1, \ldots, N, \nonumber \\
\overline{M}(t,\psi_x) \tri & \sum_{i=1}^m s_i^*(t) \psi_x(t) s_i(t) \equiv \left[ \begin{array}{cccc} \overline{M}_{11} & \overline{M}_{12} & \ldots & \overline{M}_{1N} \\
\overline{M}_{21} & \overline{M}_{22} & \ldots& \overline{M}_{2N} \\
\ldots & \ldots & \ldots &\ldots \\ 
\overline{M}_{N1} & \overline{M}_{N2} & \ldots & \overline{M}_{NN} \end{array} \right](t,\psi_x),  \nonumber 
\end{align}
where $\overline{M}_{ij} \in {\cal L}({\mathbb R}^{d_j}, {\mathbb R}^{d_i}), \: i,j=1,2, \ldots, N.$
Then from the optimal strategies under {\bf GNF} one obtains
 \begin{align}
u_t^{i,o}=&- \Big\{  R_{ii}(t) + {\mathbb E} \Big[ \overline{M}_{ii}(t,\psi_x^o(t))  | {\cal G}_{0,t}^i\Big]\Big\}^{-1} \Big\{ m^i(t) +{\mathbb E} \Big[   \sum_{j=1}^n E_{ij}(t)x^{j,o}(t) +  \overline{\Lambda^i}(t,x^o(t)\psi_x) | {\cal G}_{0,t}^i\Big]  \nonumber \\
& + \sum_{j=1,j \neq i}^N  \Big(  (R_{ij}+  {\mathbb E} \Big[ \overline{M}_{ij}(t,\psi_x(t)) u_t^{j,o}     | {\cal G}_{0,t}^i \Big]    \nonumber     \\
  & + B^{(i),*}(t)   {\mathbb  E} \Big[  \psi^{o}(t) |{\cal G}_{0,t}^i\Big] \Big\}, \hst  {\mathbb P}|_{ {\cal G}_{0,t}^i}-a.s., \hso i=1,2, \ldots, N .\label{nr4}
\end{align}
Another important observations is the following.

\begin{description}
\item[(O5):] The expressions of the optimal team strategies can be written in a fixed point form. This is described next for the case {\bf LQF}.
\end{description}

\noi {\bf Case LQF with $E \neq 0, m \neq 0$.} \\
For a  team game of linear-quadratic form (with $E, m$ non-zero) then from the previous optimal strategies one obtains
\begin{align}
u_t^{i,o}=&-  R_{ii}^{-1}(t) \Big\{  \sum_{j=1}^n E_{ij}(t) {\mathbb E}\Big(x^{j,o}(t)| {\cal G}_{0,t}^{i} \Big) +m^i(t)+ \sum_{j=1,j \neq i}^N  R_{ij}(t)  {\mathbb E} \Big( u_t^{j,o} | {\cal G}_{0,t}^i \Big)    \nonumber     \\
  & +   B^{(i),*}(t)      {\mathbb  E} \Big( \psi^{o}(t) |{\cal G}_{0,t}^i\Big) \Big\}, \hst  {\mathbb P}|_{ {\cal G}_{0,t}^i}-a.s., \hso i=1,2, \ldots, N .\label{n4r}
\end{align}
Note that (\ref{n4r}) can be put in the form of fixed point  matrix equation with random coefficients as follows. Define 
\begin{align}
\widehat{u^{i,j,o}}(t) \tri & {\mathbb E}\Big( u_t^{i,o} | {\cal G}_{0,t}^j\Big), \hso \widehat{x^{i,j,o}}(t) \tri {\mathbb E}\Big( x^{i,o}(t) | {\cal G}_{0,t}^j\Big), \nonumber \\
\widehat{u^{i,o}}(t) \tri & Vector\{ \widehat{u^{1,i,o}}(t), \ldots,  \widehat{u^{N,i,o}}(t)\},  \hso  \widehat{x^{i,o}}(t) \tri  Vector\{ \widehat{x^{1,i,o}}(t), \ldots,  \widehat{x^{N,i,o}}(t)\}, \hso      i,j=1 \ldots, N, \nonumber \\
\widehat{{ u}^o}(t) \tri &  Vector\{ \widehat{u^{1,o}}(t), \ldots,  \widehat{u^{N,o}}(t)\}, \hso \widehat{{ x}^o}(t) \tri  Vector\{ \widehat{x^{1,o}}(t), \ldots,  \widehat{x^{N,o}}(t)\},
\end{align}
\begin{align}
\widehat{\psi^o}(t) \tri & Vector\{ {\mathbb E}\Big(\psi^{o}(t) | {\cal G}_{0,t}^1\Big), \ldots,  {\mathbb E}\Big(\psi^{o}(t) | {\cal G}_{0,t}^N\Big) \} , \nonumber \\
 R^{[i]}(t)=& \left[ \begin{array}{c} R_{i1}(t), \ldots, R_{iN}(t) \end{array} \right],   \hso E^{[i]}(t)= \left[ \begin{array}{c} E_{i1}(t), \ldots, E_{iN}(t) \end{array} \right], \hso i=1, \ldots, N.  
\end{align}
Taking expectation of both sides of (\ref{n4r}) with respect to ${\cal G}_{0,t}^i$ 
 then (\ref{n4r}) is written in terms of linear equation with random coefficients as follows.
\begin{align}
 \diag\{R^{[1]}(t), \ldots, R^{[N]}(t)\} \widehat{{ u}^o}(t) +& \diag\{E^{[1]}(t), \ldots, E^{[N]}(t)\} \widehat{{x}^o}(t) \nonumber \\
 &+ \diag\{B^{(1),*}(t), \ldots, B^{(N),*}(t) \} \widehat{{\bf \psi}^o}(t)+ m(t) =0. \label{fp1}
 \end{align}
Clearly,  (\ref{fp1}) can be solved via  fixed point methods, provided we determine the estimates $\widehat{x^o}(t), \widehat{\psi^o}(t)$. In the next subsection we determine the estimate $\widehat{x^o}(t)$,  and also show that $\widehat{\psi^o}(t)$ can be expressed in terms of the estimates  $\widehat{x^o}(t), \widehat{u^o}(t)$.  \\

\noi We conclude this section by observing that the optimal team strategies involve conditional expectations with respect to the DMs information structures. These conditional expectations can be simplified considerably by allowing signaling between the different DMs.

\subsection{Team Games of Normal Form: Explicit Expressions of Adjoint Processes}
\label{nfg}
In this section we concentrate on Normal Form (and Linear-Quadratic Form) games, and we derive explicit expressions for the adjoint processes of $\psi^o(\cdot),Q^o(\cdot)$ as a functional of $x^o(\cdot), u^o(\cdot)$. Note that this is a necessary step before one proceeds with the computation of the explicit form of the optimal decentralized strategies, or the computation of them via fixed point methods as in (\ref{fp1}). \\
For  a game of Normal Form the  Hamiltonian system of equations are the following. 
\begin{align}
{\cal H}(t,x,\psi, Q, u) =& \la A(t)x+b(t) +Bu, \psi\ra  +tr \Big(Q^* \sigma(t,x,u)\Big) \nonumber \\
&+ \frac{1}{2} \la x, H(t)x\ra  + \frac{1}{2} \la u, R(t) u\ra + \la x, F(t) \ra + \la u, E(t)x\ra + \la u, m(t)\ra, \label{ex3} 
\end{align}
where $\sigma$ is given by (\ref{n1a}). 
The derivative of the Hamiltonian with respect to $u$ is 
\bea
{\cal H}_u(t,x,\psi, Q, u) = B^*(t) \psi+  R(t)u+ E(t)x + m(t) + \sum_{i=1}^m s_i^*(t) Q^{(i)}(t) . \label{ex3a}
\eea

\noi Let $(x^o(\cdot), \psi^o(\cdot), Q^o(\cdot))$ denote the solutions of the Hamiltonian system, corresponding to the optimal control $u^o$, then 
\begin{align}
 dx^o(t)=& A(t)x^o(t)dt+ b(t)dt  + B(t)u_t^o dt + \sum_{i=1}^m \Big(\kappa_i(t) x \nonumber \\
&+ s_i(t) u_t^o\Big)dW_i(t)   +   G(t)dW(t), \hst x^o(0)=x_0, \label{ex5n} \\ \nonumber \\
  d\psi^o(t)=& -A^*(t)\psi^o(t)dt   - H(t) x^o(t) dt -F(t)dt -E^*(t) u_t^odt \nonumber \\
  & -V_{Q^o}(t) dt + Q^o(t) dW(t), \hst \psi^o(T)=M(T) x^o(T)+ N(T), \label{ex6n}\\
V_{Q^o}(t) =& \sum_{i=1}^m \kappa_i^*(t) Q^{(i),o}(t), \label{vq}
\end{align}
Next, we find the form of the solution of the adjoint equation (\ref{ex6n}) (and also identify the martingale term in (\ref{ex6n}) via an alternative method to Remark~\ref{mart}). Let $\{\Phi(t,s): 0\leq s \leq t \leq T\}$ denote the transition operator of $A(\cdot)$ and $\Phi^*(\cdot, \cdot)$ that of the adjoint $A^*(\cdot)$ of $A(\cdot)$. Then we have  the identity $\frac{\partial}{\partial s} \Phi^*(t,s) = -A^*(s) \Phi^*(t,s), 0 \leq s \leq t \leq T$. One can verify by differentiation that the solution $\{ \psi^o(t): t \in [0,T]\}$ of (\ref{ex6n}),  is given by  
\begin{align}
\psi^o(t)=& \Phi^*(T,t)M(T) x^o(T) + N(T)+ \int_{t}^T \Phi^*(s,t) \Big\{ H(s) x^o(s) ds +F(s)ds + E^*(s) u_s^o  ds \nonumber \\
&+ V_{Q^o}(s)ds - Q^o(s) dW(s) \Big\}. \label{ex9nn}
\end{align}  
 Since for any control policy, $\{x^o(s): 0\leq t \leq s \leq T\}$ is uniquely determined from (\ref{ex5n}) and its current value $x^o(t)$, then (\ref{ex9nn}) can be expressed via 
\bea
\psi^o(t)=\Sigma(t) x^o(t)+ \beta^o(t) , \hst t \in [0,T], \label{ex15g}
\eea
where $\Sigma(\cdot), \beta^o(\cdot)$ determine the operators to the one expressed via (\ref{ex9nn}). \\
Next, we determine the operators $(\Sigma(\cdot), \beta^o(\cdot))$. Differentiating both sides of (\ref{ex15g}) and using (\ref{ex5n}), (\ref{ex6n}) yields
\begin{align}
-A^*(t)\psi^o(t)dt   &- H(t) x^o(t) dt -F(t)dt-E^*(t)u_t^o dt  -V_{Q^o}(t) dt + Q^o(t) dW(t) \nonumber \\
 =& \dot{\Sigma}(t)x^o(t) dt +  \Sigma(t) \Big\{ A(t)x^o(t)dt +b(t)dt+ B(t)u_t^o dt \nonumber \\
  +&\sum_{i=1}^m \Big( \kappa_i(t) x^o(t)+ s_i(t) u_t^o\Big) dW_i(t)   + G(t)dW(t)\Big\} + d\beta^o(t). \label{ex16g}
\end{align}
 By matching the intensity of the martingale terms $\{\cdot\} dW(t)$ in (\ref{ex16g}) we obtain
 \begin{align}
 Q^{i,o} (t)=& \psi_x^o(t) \sigma^{(i)}(t,x^o(t),u_t^o)= \Sigma(t)\Big(\kappa_i(t) x^o(t)+ s_i(t) u_t^o +  G^{(i)}(t)\Big), \: t \in [0,T], \: i=1, \ldots, m, \label{ex16gb}
 \end{align} 
 and by (\ref{vq}) we also obtain
 \begin{align}
 V_{Q^o}(t)= \sum_{i=1}^m \kappa_i^*(t) \Sigma(t) \Big(\kappa_i(t) x^o(t)+ s_i(t) u_t^o +  G^{(i)}(t)\Big), \hst t \in [0,T]. \label{ex16gbb} 
  \end{align}
Clearly, $Q^o$ given by (\ref{ex16gb}) is precisely the one predicted by Remark~\ref{mart}. \\
Substituting the claimed relation (\ref{ex15g}) into (\ref{ex16g}) we obtained the identity 
\begin{align}
\Big\{-A^*(t) \Sigma(t)   &-\Sigma(t) A(t) - H(t) -\sum_{i=1}^m \kappa_i^*(t) \Sigma(t) \kappa_i(t) -\dot{\Sigma}(t) \Big\} x^o(t) dt  +\sum_{i=1}^m Q^{(i),o}(t) dW_i(t) \nonumber \\
 =&  A^*(t) \beta^o(t)dt + \Sigma(t) b(t)+ \Sigma(t)  B(t)u_t^o dt + \sum_{i=1}^m \kappa_i^*(t) \Sigma(t) \Big(s_i(t) u_t^o + G^{(i)}(t)\Big)dt \nonumber \\
 &+ \Sigma(t) \sum_{i=1}^m \Big( \kappa_i(t) x^o(t) + s_i(t) u_t^o + G^{(i)}(t)\Big)dW_i(t) + F(t)dt +E^*(t)u_t^o + d\beta^o(t). \label{ex16ga}
\end{align}
Therefore, from (\ref{ex16ga}), (\ref{ex16gb}) we deduce
 \begin{align}
\dot{\Sigma}(t) +&A^*(t) \Sigma(t)   + \Sigma(t) A(t) + \sum_{i=1}^m \kappa_i^*(t)\Sigma(t) \kappa_i(t) + H(t)=0,  \hst \Sigma(T)=M(T),  \label{exg17} \\
\dot{ \beta}^o(t) &+A^*(t) \beta^o(t)+ \Sigma(t) b(t) +F(t) +\Sigma(t)  B(t)u_t^o dt+ E^*(t) u_t^o\nonumber \\
& + \sum_{i=1}^m \kappa_i^*(t) \Sigma(t) \Big(s_i(t) u_t^o + G^{(i)}(t)\Big) =0, \hst \beta^o(T)=N(T). \label{ex18g}
\end{align}
 The closed form expressions of the adjoint processes $(\psi^o(\cdot), Q^o(\cdot))$ of this section are required in order to explicitly compute the closed form expression of the optimal decentralized strategies or apply  fixed point methods via (\ref{fp1}) (in addition to solving centralized problems).\\ 
 Next we find  the optimal strategy assuming centralized information structure for each DM, and  then we determine the optimal strategies assuming  decentralized  information structures for each DM.  The reason we pursue centralized strategies is to gain additional insight into its differences when compared to decentralized strategies, both in the procedure and the amount of complexity involved in implementing centralized versus decentralized strategies.

\subsection{Centralized Information Structure: NF and LQF}
\label{FIP}
First, we consider a centralized information structure and we compute the optimal strategy for team games of Normal and Linear-Quadratic forms. For any $t \in [0,T]$ the  information structure  ${\cal G}_{0,t}^{x^u} \tri {\cal G}_{0,t}^{x^{1,u}} \bigvee  {\cal G}_{0,t}^{x^{2,u}} \ldots \bigvee  {\cal G}_{0,t}^{x^{N,u}} $    is available to all DMs and it is the $\sigma-$algebra ${\cal G}_{0,t}^{x^u} \tri \sigma\{(x^1(s), x^2(s), \ldots, x^N(s)): 0 \leq s \leq t\}$ (we assume a strong formulation so the information depends on $u$). If instead, we consider nonanticipative centralized information structure ${\cal G}_{0,t}^{x(0), W} \tri {\cal G}_{0,t}^{x^1(0), W^1} \bigvee {\cal G}_{0,t}^{x^2(0), W^2} \ldots \bigvee {\cal G}_{0,t}^{x^N(0), W^N}$ then the final results are the same.
 This is a common (centralized) full  information structure decision strategy hence, the optimal decision $\{u_t^o: 0\leq t \leq T\}$ is found via 
\bea
{\mathbb E} \Big\{ {\cal H}_u (t,x^o(t),\psi^o(t),Q^o(t),u_t^o) |{\cal G}_{0,t}^{x^o}\Big\} =0, \hst a.e.t \in [0,T], \hso {\mathbb P}|_{{\cal G}_{0,t}^{x^o}}-a.s. \label{ex4f}
\eea
where $(x^o(\cdot), \psi^o(\cdot), Q^o(\cdot))$ are solutions of the Hamiltonian system (\ref{ex5n}), (\ref{ex6n}) corresponding to $u^o$. Since $(\psi^o(\cdot), Q^o(\cdot))$ are given by  (\ref{ex15g}) and (\ref{ex16gb}), respectively, all we need to do is to determine $u^o$ as a functional of $x^o$.\\
We show the following claims.\\

\noi {\bf LQF.} When the system dynamics and pay-off are of Linear-Quadratic Form, the optimal centralized strategy is given by
\bea
u_t^o = -R^{-1}(t) B^*(t) K(t) {x^o}(t) , \hst t \in [0,T], \label{ex16f}
\eea
where the operator $K(t) \in {\cal L}({\mathbb R}^n, {\mathbb R}^n)$ is the symmetric positive semidefinite solution of the differential equation
\begin{align}
& \dot{K}(t) + A^*(t) K(t)    + K(t) A(t)
 -K(t) B(t) R^{-1}(t) B^*(t) K(t)   +H(t)=0, \hst t \in [0,T),   \label{ex20} \\
& K(T)=M(T) . \label{ex21}
\end{align}

\noi {\bf NF with $E=0$.} When the system dynamics and pay-off are of Normal Form (with $E=0$), the optimal centralized strategy is given by
\begin{align}
u_t^o =& -\Big(R(t) + \sum_{i=1}^m s_i^*(t) K(t) s_i(t)\Big)^{-1} \Big\{ \Big( B^*(t) K(t)+ \sum_{i=1}^m s_i^*(t) K(t) \kappa_i(t) \Big)x^o(t) \nonumber \\
& +m(t) + B^*(t) r(t)\Big\}, \hst t \in [0,T], \label{ex16fn}
\end{align}
where the operator $K(t) \in {\cal L}({\mathbb R}^n, {\mathbb R}^n)$ is symmetric positive semidefinite, and $r(t) \in {\mathbb R}^n$, and they  are  solutions of the differential equations
\begin{align}
 &\dot{K}(t) + A^*(t) K(t)    + K(t) A(t)  + \sum_{i=1}^m \kappa_i^*(t) K(t) \kappa_i(t)  + H(t) \nonumber \\
&- \Big( K(t) B(t) + \sum_{i=1}^m \kappa_i^*(t) K(t) s_i(t) \Big) \Big( R(t) + \sum_{i=1}^m s_i^*(t)K(t) s_i(t)\Big)^{-1}\Big( B^*(t) K(t) \nonumber \\
&+ \sum_{i=1}^m s_i^*(t) K(t)\kappa_i(t) \Big)=0, \hst t \in [0,T),   \label{ex20nn} \\
 &K(T)=M(T), \label{ex21nn}
\end{align}
\begin{align}
 &\dot{r}(t)+ \Big\{ A^*(t) - \Big( K(t) B(t) + \sum_{i=1}^m \kappa_i^*(t) K(t) s_i(t) \Big)\Big( R(t) + \sum_{i=1}^m s_i^*(t)K(t) s_i(t)\Big)^{-1}B^*(t) \Big\} \nonumber \\
&+ F(t) + K(t) b(t) \nonumber \\
&- \Big( K(t) B(t) + \sum_{i=1}^m \kappa_i^*(t) K(t) s_i(t) \Big)\Big( R(t) + \sum_{i=1}^m s_i^*(t)K(t) s_i(t)\Big)^{-1}m(t) , \hst t \in [[0,T), \label{ex20n} \\
&r(T)= N(T). \label{ex21n}
\end{align}
Next, we  verify the claim stated under {\bf LQF} and we leave the claim stated under {\bf NF} to the reader since its derivation is similar.\\

\noi {\bf Derivation of LQF Solution.}  From (\ref{ex4f}) the optimal strategy is
\bea
u_t^o = -R^{-1}(t) B^*(t) {\mathbb E} \Big\{ \psi^o(t) |{\cal G}_{0,t}^{x^o}\Big\}, \hst t \in [0,T],  \label{ex8fa}
\eea
where  $(x^o(\cdot), \psi^o(\cdot), Q^o(\cdot))$ denote the solutions of the following Hamiltonian system, corresponding to the optimal control $u^o$
\begin{align}
 dx^o(t) =& A(t)x^o(t)dt + B(t)u_t^o dt + G(t)dW(t), \hst x^o(0)=x_0 \label{ex5} \\  
  d\psi^o(t)=& -A^*(t)\psi^o(t)dt   - H(t) x^o(t) dt  -V_{Q^o}(t) dt+ Q^o(t) dW(t), \hst \psi^o(T)=M(T) x^o(T) \label{ex6}\\
  V_{Q^o}(t)=&0, \hst Q^o(t)= \Sigma(t) G(t). \label{vq1}
\end{align}
  Then $\{ \psi^o(t): t \in [0,T]\}$ is given by (\ref{ex9nn}) with $F=0, E=0$, hence
\begin{align}
\psi^o(t)=\Phi^*(T,t)M(T) x^o(T) + \int_{t}^T \Phi^*(s,t) \Big\{ H(s) x^o(s) ds  - Q^o(s) dW(s)  \Big\}. \label{ex9}
\end{align}  
For any admissible decision $u$ and corresponding $(x(\cdot), \psi(\cdot))$ define their filtered versions by 
\bes
\overline{x}(t) \tri  {\mathbb E} \Big\{ x(t) |{\cal G}_{0,t}^{x}\Big\}= x(t), \hst \overline{\psi}(t) \tri  {\mathbb E} \Big\{ \psi(t) |{\cal G}_{0,t}^{x}\Big\}, \hst   \overline{u}(t) \tri  {\mathbb E} \Big\{ u_t |{\cal G}_{0,t}^{x}\Big\}= u_t, \hst    t \in [0,T],
\ees
and their predicted versions by 
\bes
\overline{x}(s,t) \tri  {\mathbb E} \Big\{ x(s) |{\cal G}_{0,t}^{x}\Big\}, \hst \overline{\psi}(s,t) \tri  {\mathbb E} \Big\{ \psi(s) |{\cal G}_{0,t}^{x}\Big\}, \hst  \overline{u}(s,t) \tri  {\mathbb E} \Big\{ u_s |{\cal G}_{0,t}^{x}\Big\},   \hst  0 \leq  t \leq s \leq T.
\ees
From (\ref{ex8fa})  the optimal strategy is
\bea
u_t^o \equiv -R^{-1}(t) B^*(t)  \overline{\psi^o}(t) , \hst t \in [0,T]. \label{ex8f}
\eea
Taking conditional expectations on both sides of (\ref{ex9}) with respect to ${\cal G}_{0,t}^{x^o}$ yields
\begin{align}
\overline{\psi^o}(t)= \Phi^*(T,t)M(T) \overline{x^o}(T,t) + \int_{t}^T \Phi^*(s,t)  H(s) \overline{x^o}(s,t) ds,   \hst t \in [0,T], \label{ex10f}
\end{align}  
where we utilized the fact that   $${\mathbb E} \Big\{ \int_{t}^T \Phi^*(s,t)   Q^o(s) dW(s)|{\cal G}_{0,t}^{x^o} \Big\} ={\mathbb E} \Big\{  {\mathbb E} \Big\{  \int_{t}^T \Phi^*(s,t)   Q^o(s) dW(s) |{\mathbb  F}_{0,t} \Big\}  |{\cal G}_{0,t}^{x^o} \Big\}  =    0.$$
The predictor version of $x^o(\cdot)$ is obtained from (\ref{ex5}) utilizing the fact that the last right hand side of this equation is a stochastic integral with respect to Brownian motion, hence
\begin{align}
& d \overline{x^o}(s,t)= A(s) \overline{x^o}(s,t)dt+ B(s)\overline{u^o}(s,t), \hst t <s \leq T, \label{ex10ff} \\
& \overline{x^o}(t,t)= \overline{x^o}(t)=x^o(t), \hst t \in [0,T). \label{ex10fff} 
 \end{align}
Since for any  policy and hence for the optimal $u^o$, $\{\overline{x^o}(s,t): 0\leq t \leq s \leq T\}$ is uniquely determined from (\ref{ex10ff})  and the current value $\overline{x^o}(t,t)={x^o}(t)$ via (\ref{ex5}), then (\ref{ex10f}) can be expressed via 
\bea
\overline{\psi^o}(t)=K(t) \overline{x^o}(t)=K(t) x^o(t), \hst t \in [0,T], \label{ex15f}
\eea
where $K(\cdot)$ determines the operator to the one expressed via (\ref{ex10f}).  Substituting (\ref{ex15f}) into (\ref{ex8f}) we obtain (\ref{ex16f}).
Let $\{\Psi_K(t,s): 0\leq s \leq t \leq T\}$ denote the transition operator of $A_K(t) \tri \Big(A(t)-B(t)R^{-1}(t) B^*(t) K(t)\Big)$ and recall that  the identities $\frac{\partial}{\partial t} \Psi_K(t,s) = A_K(t) \Psi_K(t,s), 0 \leq s \leq t \leq T$,  $\frac{\partial}{\partial t} \Psi_K(s,t) =  -\Psi_K(s,t)A_K(t), 0 \leq t \leq s \leq T$.\\
Next, we determine $K(\cdot)$. Substituting the solution of (\ref{ex10ff}), (\ref{ex10fff}), specifically, $\overline{x^o}(s,t) = \Psi_K(s,t) {x^o}(t), 0\leq t \leq s \leq T$ into (\ref{ex10f}) we have 
\begin{align}
\overline{\psi^o}(t)= \Big\{\Phi^*(T,t)M(T) \Psi_K(T,t)  + \int_{t}^T \Phi^*(s,t)  H(s) \Psi_K(s,t) ds \Big\}{x^o}(t),   \hst t \in [0,T], \label{ex17f}
\end{align}  
and thus $K(\cdot)$ is identified by the operator  
\bea
K(t) \tri \Phi^*(T,t)M(T) \Psi_K(T,t)  + \int_{t}^T \Phi^*(s,t)  H(s) \Psi_K(s,t) ds,   \hst t \in [0,T]. \label{ex18f}
\eea
Differentiating both sides of (\ref{ex18f})  yields the following differential equation for $K(\cdot)$.
\begin{align}
\dot{K}(t) &=  \frac{\partial  }{\partial t} \Phi^*(T,t)  M(T) \Psi_K(T,t)  +\Phi^*(T,t)M(T)  \frac{\partial  }{\partial t}   \Psi_K(T,t)    -H(t)   \nonumber \\
&  +       \int_{t}^T    \frac{\partial  }{\partial t}  \Phi^*(s,t)  H(s)   \Psi_K(s,t) ds  +       \int_{t}^T    \Phi^*(s,t)  H(s)\frac{\partial  }{\partial t}    \Psi_K(s,t) ds \nonumber \\
 &=-A^*(t) \Phi^*(T,t)  M(T) \Psi_K(T,t)  -\Phi^*(T,t)M(T)     \Psi_K(T,t)A_K(t)    -H(t)   \nonumber \\
&  -       \int_{t}^T    A^*(t)  \Phi^*(s,t)  H(s)   \Psi_K(s,t) ds  -       \int_{t}^T    \Phi^*(s,t)  H(s)    \Psi_K(s,t) A_K(t) ds.  \label{ex19}
\end{align}
Using (\ref{ex18f}) in the previous equations we obtain the  matrix differential equation (\ref{ex20}), (\ref{ex21}).\\
An alternative approach is to utilize (\ref{ex15g}), (\ref{exg17}), (\ref{ex18g}) (with $\kappa_i, b, F, E, s_i=0$) which implies 
\bea
\psi^o(t)= \Sigma(t) x^o(t) + \int_{t}^T \Phi^*(s,t) \Sigma(s) B(s) u_s^o ds. \label{ex22}
\eea
 Then replace $u^o(\cdot)$ in (\ref{ex22}) by (\ref{ex8f}) and  take conditional  expectation  to obtain 
 \bea
\overline{\psi}^o(t)= \Sigma(t) x^o(t) - \int_{t}^T \Phi^*(s,t) \Sigma(s) B(s) R^{-1}(s) B^*(s) {\mathbb E} \Big\{ \psi^o(s) | {\cal G}_{0,t}^{x^o} \Big\} ds. \label{ex23}
\eea
Next, assume $\overline{\psi}^o(t)= K(t) x(t)$, for some $K(\cdot)$, and then  substitute this in (\ref{ex23}) to obtain 
 \bea
K(t)= \Sigma(t) - \int_{t}^T \Phi^*(s,t) \Sigma(s) B(s) R^{-1}(s) B^*(s) K(s) \Psi_K(s,t) ds. \label{ex24}
\eea
By utilizing the equation for $\Sigma(\cdot)$ it can be shown that (\ref{ex24}) is a solution of (\ref{ex20}), (\ref{ex21}).\\
The previous calculations demonstrate how to compute the optimal strategy when both decision variables are based on centralized  information structures, and its is precisely the optimal strategy obtained via variety of other methods in the literature. \\
Note that certain computations presented above are also required to compute an expression for the estimate $\widehat{\psi^o}(t)$ entering the fixed point equation (\ref{fp1}). \\
Finally, one can verify that the necessary conditions of optimality of Theorem~\ref{theorem7.1} utilized to derive the above optimal strategy are also sufficient. Specifically, in view of Theorem~\ref{theorem7.1} it suffices to show convexity of $\varphi(x)=\frac{1}{2} \la x,M(T)x\ra + \la x, N(T) \ra$ and joint convexity of the Hamiltonian ${\cal H}(t,x,\psi,Q,u)$  in $(x,u)$. Since $M(T)\geq 0$ then $\varphi(x)$ is convex, and since $H(\cdot)\geq 0, R(\cdot)>0$ then  ${\cal H}(t,x,\psi,Q,u)$ is convex in $(x,u)$. \\

\subsection{Decentralized Information Structures for LQF}
\label{dec}

In this section we invoke the minimum principle to compute the optimal strategies for team games of Linear-Quadratic Form. We consider decentralized strategies based on 1) nonanticipative information structures, and 2) feedback information structures. Without loss of generality we assume  the  distributed stochastic dynamical decision systems consists of  an interconnection of two subsystems,  each governed by a linear stochastic differential equation with coupling. The generalizations to an arbitrary number of interconnected subsystems will be given as a corollary.\\
Consider the distributed dynamics described below.

\begin{align}
\mbox{\bf Subsystem Dynamics 1:}&  \nonumber \\
  dx^1(t)=& A_{11}(t)x^1(t)dt + B_{11}(t)u_t^1dt  + G_{11}(t) dW^1(t)   \nonumber \\
  &+ A_{12}(t)x^2(t)dt  + B_{12}(t)u_t^2 dt , \hst   x^1(0)=x^1_0, \hso   t \in (0,T],   \label{ex30} \\ \nonumber \\
\mbox{\bf Subsystem Dynamics 2:}&  \nonumber  \\
dx^2(t)=&  A_{22}(t)x^2(t)dt  + B_{22}(t)u_t^2 dt + G_{22}(t) dW^2(t)  \nonumber \\
&+ A_{21}(t)x^1(t)dt + B_{21}u_t^1dt , \hst x^2(0)=x^2_0, \hso  t \in (0,T] \label{ex31} 
\end{align}
For any $t \in [0,T]$ the information structure       of  $u_t^1$ of subsystem $1$ is  the $\sigma-$algebra ${\cal G}_{0,t}^{{1}}$, and 
  information structure of  $u_t^2$ of subsystem $2$ is the $\sigma-$algebra ${\cal G}_{0,t}^{{2}}$. These  information structures are  defined  shortly.\\

{\bf Pay-off  Functional: }\\

\begin{align}
J(u^1,u^2) =& \frac{1}{2}  {\mathbb E} \Big\{ \int_{0}^T \Big[ \la \left( \begin{array}{c} x^{1}(t) \\ x^{2}(t) \end{array} \right), H(t)   \left(\begin{array}{c} x^1(t) \\ x^2(t) \end{array} \right)      \ra
+ \la \left( \begin{array}{c} u_t^{1}(t) \\ u_t^{2}(t) \end{array} \right), R(t)   \left(\begin{array}{c} u_t^1(t) \\ u_t^2(t) \end{array} \right) \ra         \Big]dt \nonumber \\
& + \la   \left( \begin{array}{c} x^{1}(T) \\ x^{2}(T) \end{array} \right), M(T)   \left(\begin{array}{c} x^1(T) \\ x^2(T) \end{array} \right) \ra            \Big\}.  \label{ex34}
\end{align}
We assume that the initial condition $x(0)$, the system Brownian motion $\{W(t): t \in [0,T]\}$, and the observations Brownian motion $\{B^1(t): t \in [0,T]\}$, and $\{B^2(t): t \in [0,T]\}$ are mutually independent and $x(0)$ is Gaussian $({\mathbb E}(x(0)), Cov(x(0)))=(\bar{x}_0, P_0).$

\noi Define the augmented variables by 
\begin{align}
x \tri \left(\begin{array}{c} x^1 \\ x^2 \end{array} \right), \hso u \tri \left(\begin{array}{c} u^1 \\ u^2 \end{array} \right),  \hso \psi \tri \left(\begin{array}{c} \psi^1 \\ \psi^2 \end{array} \right), \hso Q \tri \left(\begin{array}{c} Q^1 \\ Q^2 \end{array} \right), \hso
W \tri & \left(\begin{array}{c} W^1 \\ W^2 \end{array} \right)
\end{align}
and matrices by
\begin{align}
 A \tri & \left[\begin{array}{cc} A_{11}  & A_{12} \\ A_{21} & A_{22} \end{array} \right], \: B \tri \left[\begin{array}{cc} B_{11}  & B_{12} \\ B_{21} & B_{22} \end{array} \right], \:
  B^{(1)} \tri \left[ \begin{array}{c} B_{11}  \\ B_{21} \end{array} \right], \: B^{(2)} \tri \left[ \begin{array}{c} B_{12} \\ B_{22} \end{array} \right],  \: G\tri \left[\begin{array}{cc} G_{11}  & 0 \\ 0 & G_{22} \end{array} \right]. \nonumber \\ \nonumber
\end{align}
Let $(x^o(\cdot), \psi^o(\cdot),  Q^o(\cdot))$ denote the solutions of the Hamiltonian system, corresponding to the optimal control $u^o$, then 
\begin{align}
 dx^o(t) =& A(t)x^o(t)dt + B(t)u_t^o dt + G(t)dW(t), \hst x^o(0)=x_0, \label{nex5} \\ \nonumber \\
  d\psi^o(t)=& -A^*(t)\psi^o(t)dt   - H(t) x^o(t) dt  -V_{Q}^o(t) dt 
+ Q^o(t) dW(t), \hst \psi^o(T)=M(T) x^o(T), \label{nex6} \\
V_{Q^o}(t)=&0, \hst Q^o(t)=\Sigma(t) G(t), \hst \psi^o(t)= \Sigma(t) x^o(t) + \beta^o(t), \label{mv2}
\end{align}
where $\Sigma(\cdot), \beta^o(\cdot)$ are given by (\ref{exg17}), (\ref{ex18g}) with $s_i, \kappa_i, b, F, E=0$.  
 The optimal decisions $\{(u_t^{1,o}, u_t^{2,o}): 0 \leq t\leq T\}$ are obtained from (\ref{ex3}) with $\sigma(t,x,u)=G(t), b=0, F=0, E=0, m=0$, and they are given by 
\begin{align}
{\mathbb E} & \Big\{ {\cal H}_{u^1} (t,x^{1,o}(t), x^{2,o}(t), \psi^{1,o}(t),\psi^{2,o}(t), Q^{1,o}(t),Q^{2,o}(t), u_t^{1,o}, u_t^{2,0}) |{\cal G}_{0,t}^1\Big\} =0, \nonumber \\
& \hst a.e.t \in [0,T], \hso {\mathbb P}|_{{\cal G}_{0,t}^{1}}-a.s. \label{ex35}
\end{align}
\begin{align}
{\mathbb E} & \Big\{ {\cal H}_{u^2} (t,x^{1,o}(t), x^{2,o}(t), \psi^{1,o}(t),\psi^{2,o}(t), Q^{1,o}(t),Q^{2,o}(t), u_t^{1,o}, u_t^{2,0}) |{\cal G}_{0,t}^2\Big\} =0, \nonumber \\
& \hst a.e.t \in [0,T], \hso {\mathbb P}|_{{\cal G}_{0,t}^{2}}-a.s. \label{ex36}
\end{align}
From (\ref{ex35}), (\ref{ex36}) the optimal decisions are 
\bea
u_t^{1,o} = -R_{11}^{-1}(t)   B^{(1),*}(t)  {\mathbb E} \Big\{ \psi^{o}(t) |{\cal G}_{0,t}^1\Big\} -R_{11}^{-1}(t)R_{12}(t) {\mathbb E} \Big\{u_t^{2,o} | {\cal G}_{0,t}^1\Big\}  ,  \hst t \in [0,T]. \label{ex41}
\eea
\bea
u_t^{2,o} = -R_{22}^{-1}(t) B^{(2),*}(t)  {\mathbb E} \Big\{ \psi^{o}(t) |{\cal G}_{0,t}^2\Big\} -R_{22}^{-1}(t)R_{21}(t) {\mathbb E} \Big\{u_t^{1,o} | {\cal G}_{0,t}^2\Big\}.  \hst t \in [0,T]. \label{ex42}
\eea
From the previous expressions we notice the following. 

\begin{description}
\item[(O6):]
The optimal strategies (\ref{ex41}), (\ref{ex42})  illustrate  the signaling between $u^1$ and $u^2$, which is   facilitated by the coupling in the pay-off via $R(\cdot)$, and the coupling in the state dynamics of $x^1$ and $x^2$ via $\psi^o(t)=\Sigma(t) x^o(t)+ \beta^o(t)$. Clearly,  $u^{1,o}$  estimates the optimal decision of  subsystem 2,  $u^{2,o}$, and  the adjoint processes $\psi^{o}$ from its observations, and vice-versa. This coupling is simplified if we consider a simplified model of dynamical coupling between subsystems $x^1, x^2$ and/or  nested information structures, i.e.,   ${\cal G}_{0,t}^2 \subset  {\cal G}_{0,t}^1$. Moreover, if we consider  no coupling through the pay-off, i.e., a diagonal $R(\cdot)$, then the second right hand side terms in (\ref{ex41}), (\ref{ex42}) will be zero, implying that the signaling between $u^{1,o}, u^{2,o}$ is done via the adjoint process $\psi^o$.  
\end{description}

Let $ \phi(\cdot)$ be any   square integrable and ${\mathbb F}_T-$adapted matrix-valued process or scalar-valued processes,  and define its filtered and predictor  versions by   
\bes
\pi^i (\phi)(t) \tri {\mathbb E} \Big\{ \phi(t) | {\cal G}_{0,t}^i \Big\}, \hst  \pi^i (\phi)(s,t) \tri {\mathbb E} \Big\{ \phi(s) | {\cal G}_{0,t}^i  \Big\}, \hst t \in [0,T], \hso s \geq t, \hso i=1,2.
\ees
For any admissible  decision $u$ and corresponding  $(x(\cdot), \psi(\cdot))$ define their  filter versions  with respect to ${\cal G}_{0,t}^i$ for $i=1, 2$,  by 
\bes
\pi^i(x)(t)  \tri    \left[ \begin{array}{c}  {\mathbb E} \Big\{x^{1}(t)  |  {\cal G}_{0,t}^i \Big\}   \\  {\mathbb E}  \Big\{ x^{2}(t) |  {\cal G}_{0,t}^i \Big\}  \end{array}     \right]  \equiv  \widehat{x}^i(t), \hst t \in [0,T], \hso i=1,2, 
\ees
\bes
 \pi^i(\psi)(t) \tri \left[\begin{array}{c}  {\mathbb E} \Big\{   \psi^{1}(t)  |  {\cal G}_{0,t}^i \Big\}   \\   {\mathbb E}  \Big\{ \psi^{2}(t)   |    {\cal G}_{0,t}^i\Big\}  \end{array}    \right]  \equiv {\widehat{\psi}}^i(t),
    \hst t \in [0,T], \hso i=1,2,
    \ees
\bes
\pi^i(u)(t)  \tri    \left[ \begin{array}{c}  {\mathbb E} \Big\{u_t^{1}  |  {\cal G}_{0,t}^i \Big\}   \\  {\mathbb E}  \Big\{ u_t^{2}|  {\cal G}_{0,t}^i \Big\}  \end{array}     \right]  \equiv \widehat{u}^i(t), \hst t \in [0,T] , \hso i=1,2, 
\ees
and their predictor versions by
\bes
\pi^i(x)(s,t)  \tri    \left[ \begin{array}{c}  {\mathbb E} \Big\{x^{1}(s)  |  {\cal G}_{0,t}^i \Big\}   \\  {\mathbb E}  \Big\{ x^{2}(s) |  {\cal G}_{0,t}^i\Big\}  \end{array}     \right]  \equiv \widehat{x}^i(s,t), \hst t \in [0,T], \hso s\geq t, \hso i=1,2, 
\ees
\bes
 \pi^i(\psi)(s,t) \tri \left[\begin{array}{c}  {\mathbb E} \Big\{   \psi^{1}(s)  |  {\cal G}_{0,t}^i \Big\}   \\   {\mathbb E}  \Big\{ \psi^{2}(s)   |    {\cal G}_{0,t}^i\Big\}  \end{array}    \right] \equiv \widehat{\psi}^i(s,t),
    \hst t \in [0,T], \hso s \geq t, \hso i=1,2.
\ees
\bes
\pi^i(u)(s,t)  \tri    \left[ \begin{array}{c}  {\mathbb E} \Big\{u_s^{1}  |  {\cal G}_{0,t}^i \Big\}   \\  {\mathbb E}  \Big\{ u_s^{2} |  {\cal G}_{0,t}^i \Big\}  \end{array}     \right]    \equiv \widehat{u}^i(s,t), \hst t \in [0,T], \hso s\geq t, \hso i=1,2, 
\ees
From (\ref{ex41}), (\ref{ex42}) the optimal decisions are 
\begin{align}
u_t^{1,o} \equiv  -R_{11}^{-1}(t) B^{(1),*}(t) \pi^1(\psi^o)(t)   -R_{11}^{-1}(t)R_{12}(t) {\mathbb E} \Big\{u_t^{2,o} | {\cal G}_{0,t}^1\Big\},   \hst t \in [0,T],    \label{ex41a}
\end{align}
\begin{align}
u_t^{2,o}  \equiv  -R_{22}^{-1}(t) B^{(2),*}(t) \pi^{w^2}(\psi^o)(t)-R_{22}^{-1}(t)R_{21}(t) {\mathbb E} \Big\{u_t^{1,o} | {\cal G}_{0,t}^2\Big\}, \hst t \in [0,T].  \label{ex42a}
\end{align}
The previous  optimal decisions  require  the conditional estimates \\$\{(\pi^1(\psi^o)(t),    \pi^{2}(\psi^o)(t)): 0\leq t \leq T\}$. These are obtained by  taking conditional expectations of (\ref{ex9}) giving   
\begin{align}
\pi^i(\psi^{o})(t)= \Phi^*(T,t)M(T) \pi^i(x^{o})(T,t) + \int_{t}^T \Phi^*(s,t)  H(s) \pi^i(x^{o})(s,t)ds,   \hst t \in [0,T], \hso i=1,2. \label{ex49}
\end{align}  
Before we proceed further we shall specify the information structures available to the DMs.\\

\noi {\bf Nonanticipative Information Structures.} The information structure available to $u^1$ is ${\cal G}_{0,t}^1 \tri \sigma\{ W^1(s): 0 \leq s \leq t\} \equiv {\cal G}_{0,t}^{W^1}$, and the information structure available to $u^2$ is ${\cal G}_{0,t}^2 \tri \sigma\{ W^2(s): 0 \leq s \leq t\} \equiv {\cal G}_{0,t}^{W^2}$.  Therefore, by denoting $\pi^{w^i}(\cdot)(\cdot)$ the conditional expectation with respect to ${\cal G}_{0,\cdot}^{W^i},i=1,2$,  for any admissible  decision, the filtered versions of $x(\cdot)$ based on this information structures  are given by the following  stochastic differential equations  \cite{liptser-shiryayev1977} (Theorem~8.2).
\begin{align}
d\pi^{w^1}(x)(t) =& A(t) \pi^{w^1}(x)(t)dt + B^{(1)}(t) u_t^{1} dt 
 + B^{(2)}(t) \pi^{w^1}(u^{2})(t) dt \nonumber \\
&+G_{11}(t)dW^1(t), \hso \pi^{w^1}(x)(0)=\bar{x}_0, \label{ex50}
\end{align}
\begin{align}
d\pi^{w^2}(x)(t) =& A(t) \pi^{w^2}(x)(t)dt + B^{(2)}(t) u_t^{2} dt + B^{(1)}(t)\pi^{w^2}(u^{1})(t) dt  \nonumber \\
&+G_{22}dW^2(t), \hso \pi^{w^2}(x)(0)
=\bar{x}_0. \label{ex51}
\end{align}
From the previous filtered versions of $x(\cdot)$ it is clear that  subsystem $1$ estimates the augmented state vector and the  actions of subsystem $2$ based on its own observations, namely, $\pi^{w^1}(u^{2})(\cdot)$ and subsystem $2$ estimates the augmented state vector and the actions of subsystem $1$ based on its own observations, namely, $\pi^{w^2}(u^{1})(\cdot)$.  \\

\noi For any admissible decision $u$ the predicted versions of $x(\cdot)$ are obtained from (\ref{ex50}) and (\ref{ex51}) as follows. Utilizing the identity $\pi^{w^i}(x)(s,t)={\mathbb E}  \Big\{  {\mathbb E} \Big\{ x(s)| {\cal G}_{0,s}^{W^{i} }\Big\} |  {\cal G}_{0,t}^{W^{i}} \Big\}= {\mathbb E} \Big\{ \pi^{w^i}(x)(s) | {\cal G}_{0,t}^{W^{i}} \Big\}$, for $0 \leq t \leq s \leq T$ then 
\begin{align}
&\frac{d}{ds}\pi^{w^1}(x)(s,t) = A(s) \pi^{w^1}(x)(s,t) + B^{(1)}(s) \pi^{w^1}(u^1)(s,t) +B^{(2)}(s) \pi^{w^1}(u^{2})(s,t), 
 \hst    t < s \leq T, \label{ex52} \\
&\pi^{w^1}(x)(t,t)=\pi^{w^1}(x)(t), \hst t \in [0,T), \label{ex52a} 
\end{align}
\begin{align}
&\frac{d}{d s}\pi^{w^2}(x)(s,t) = A(s) \pi^{w^2}(x)(s,t) +B^{(2)}(s) \pi^{w^2}(u^{1})(s,t)+ B^{(1)}(s) \pi^{w^2}(u^{1})(s,t), \hst  t < s \leq T, \label{ex53}  \\
&\pi^{w^2}(x)(t,t)=\pi^{w^2}(x)(t), \hst t \in [0,T).
\label{ex53a}
\end{align}
Since for a given admissible policy and observation paths, $\{\pi^{w^1}(x)(s,t): 0\leq t \leq s \leq T\}$ is  determined from (\ref{ex52}) and its current value $\pi^{w^1}(x^o)(t,t)=\pi^{w^1}(x)(t)$, and $\{\pi^{w^2}(x)(s,t): 0\leq t \leq s \leq T\}$ is  determined from (\ref{ex53}), and its current value $\pi^{w^2}(x)(t,t)=\pi^{w^2}(x)(t)$, then (\ref{ex49}) can be expressed via
\bea
\pi^{w^i}(\psi^o)(t) = K^i(t) \pi^{w^i}(x^o)(t) + r^i(t), \hst t\in [0,T], \hso i=1,2. \label{in54}
\eea
 where $K^i(\cdot), r^i(\cdot)$  determines the operators  to the one expressed via (\ref{ex49}), for $i=1,2$. Utilizing (\ref{in54})  into (\ref{ex41a}) and (\ref{ex42a}) then 
\begin{align}
u_t^{1,o} \equiv  -R_{11}^{-1}(t) B^{(1),*}(t) \Big\{K^1(t) \pi^{w^1}(x^o)(t) +r^1(t) \Big\}  -R_{11}^{-1}(t)R_{12}(t) \pi^{w^1}({u_t^{2,o}})(t),   \hst t \in [0,T],    \label{in55}
\end{align}
\begin{align}
u_t^{2,o}  \equiv  -R_{22}^{-1}(t) B^{(2),*}(t) \Big\{ K^2(t)    \pi^{w^2}(x^o)(t) +r^2(t)\Big\} -R_{22}^{-1}(t)R_{21}(t) \pi^{w^2}({u^{1,o}})(t), \hst t \in [0,T].  \label{in56}
\end{align}
Let $\{\Psi_{K^i}(t,s): 0\leq s \leq t \leq T\}$ denote the transition operator of $A_{K^i}(t) \tri \Big(A(t)-B^{(i)}(t)R_{ii}^{-1}(t) B^{(i),*}(t) K^i (t) \Big)$, for $i=1,2$. \\
 Next, we determine $K^i(\cdot), r^i(\cdot), i=1,2$. Substituting  the previous equations into (\ref{ex52}), (\ref{ex52a}) and (\ref{ex53}), (\ref{ex53a}) then
  \begin{align}
 \pi^{w^1}(x^o)(s,t)=& \Psi_{K^1}(s,t) \pi^{w^1}(x^o)(t) -\int_{t}^s \Psi_{K^1}(s,\tau) B^{(1)}(\tau)R_{11}^{-1}(\tau)B^{(1),*}(\tau) r^1(\tau)d\tau  \nonumber \\
 &-\int_t^s \Psi_{K^1}(s,\tau) B^{(1)}(\tau)R_{11}^{-1}(\tau) R_{12}(\tau) \pi^{w^1}(u^{2,o})(\tau,t)d\tau  \nonumber \\
 &+ \int_{t}^s \Psi_{K^1}(s, \tau) B^{(2)}(\tau) \pi^{w^1}(u^{2,o})(\tau,t) d\tau , \hst t \leq s \leq T, \label{pr1}
 \end{align}
  \begin{align}
 \pi^{w^2}(x^o)(s,t)=& \Psi_{K^2}(s,t) \pi^{w^2}(x^o)(t)  -\int_{t}^s \Psi_{K^2}(s,\tau) B^{(2)}(\tau)R_{22}^{-1}(\tau)B^{(2),*}(\tau) r^2(\tau)d\tau  \nonumber \\
 &-\int_t^s \Psi_{K^2}(s,\tau) B^{(2)}(\tau)R_{22}^{-1}(\tau) R_{21}(\tau) \pi^{w^2}(u^{1,o})(\tau,t)d\tau  \nonumber \\
 &+ \int_{t}^s \Psi_{K^2}(s, \tau) B^{(1)}(\tau) \pi^{w^2}(u^{1,o})(\tau,t) d\tau , \hst t \leq s \leq T. \label{pr2}
 \end{align}
Since $u_t^{1,o}$ is ${\cal G}_{0,t}^{W^1}-$measurable and $u_t^{2,o}$ is ${\cal G}_{0,t}^{W^2}-$measurable, and ${\cal G}_{0,t}^{W^1}$ and ${\cal G}_{0,t}^{W^2}$ are independent, then  $\pi^{w^1}(u^{2,o})(\tau,t)= {\mathbb E} \big(u_\tau^2\big)\equiv \overline{u^{2,o}}(\tau)$,  $\pi^{w^2}(u^{1,o})(\tau,t)= {\mathbb E} \big(u_\tau^1\big)\equiv \overline{u^{1,o}}(\tau)$,  $0 \leq t \leq \tau \leq T$. Utilizing the last observation we show in the next main theorem that the optimal DM strategies are finite dimensional (i.e., given in terms of finite number of statistics), and that each optimal strategy is linear function of  the augmented state estimate based on his information, and the average value of the other optimal strategy. The computation of the average optimal strategies can be expressed in fixed point form.

\begin{theorem}(Optimal decentralized strategies for {\bf LQF}) \\
\label{dise}
Given a {\bf LQF} game the optimal decisions $(u^{1,o}, u^{2,o})$ are given 
\begin{align}
u_t^{1,o} \equiv  -R_{11}^{-1}(t) B^{(1),*}(t) \Big\{K^1(t) \pi^{w^1}(x^o)(t) +r^1(t) \Big\}  -R_{11}^{-1}(t)R_{12}(t) \overline{u^{2,o}}(t),   \hst t \in [0,T],    \label{ex49i}
\end{align}
\begin{align}
u_t^{2,o}  \equiv  -R_{22}^{-1}(t) B^{(2),*}(t) \Big\{ K^2(t)    \pi^{w^2}(x^o)(t) +r^2(t)\Big\} -R_{22}^{-1}(t)R_{21}(t) \overline{u^{1,o}}(t), \hst t \in [0,T].  \label{ex49j}
\end{align}
where  $\pi^{w^i}(x^o)(\cdot), i=1,2$ satisfy the  linear non-homogeneous stochastic differential equations
\begin{align}
d\pi^{w^1}(x)(t) =& A(t) \pi^{w^1}(x)(t)dt + B^{(1)}(t) u_t^{1,o} dt 
 + B^{(2)}(t) \overline{u^{2,o}}(t) dt \nonumber \\
&+G_{11}(t)dW^1(t), \hso \pi^{w^1}(x)(0)=\bar{x}_0, \label{ex500}
\end{align}
\begin{align}
d\pi^{w^2}(x)(t) =& A(t) \pi^{w^2}(x)(t)dt + B^{(2)}(t) u_t^{2,o} dt + B^{(1)}(t)\overline{u^{1,o}}(t) dt  \nonumber \\
&+G_{22}dW^2(t), \hso \pi^{w^2}(x)(0)
=\bar{x}_0. \label{ex510}
\end{align}
and $\Big(K^i(\cdot), r^i(\cdot), \overline{x^o}(\cdot), \overline{u^{i,o}}(\cdot) \Big),i=1,2$ are solutions of  the ordinary differential equations (\ref{ex20pi}), (\ref{ex49f}), (\ref{ex49g}), (\ref{ex49h}), (\ref{ex101}), (\ref{ex106}) below.
\begin{align}
\dot{K}^i(t) &+ A^*(t) K^i(t)    + K^i(t) A(t)
 -K^i(t) B^{(i)}(t) R_{ii}^{-1}(t) B^{(i),*}(t) K^i(t)   \nonumber \\
 &+H(t)=0, \hst t \in [0,T), \hso i=1,2,   \label{ex20pi} \\
 K^i(T)&=M(T), \hso i=1,2 , \label{ex49f}
\end{align}
\begin{align}
\dot{r}^1(t)=&\Big\{  -A^{*}(t)  +\Phi^*(T,t) M(T)  \Psi_{K^1}(T,t) B^{(1)}(t) R_{11}^{-1}(t)B^{(1),*}(t) \nonumber \\                
&+ \Big(\int_{t}^T \Phi^*(s,t) H(s)  \Psi_{K^1}(s,t) ds \Big) B^{(1)}(t) R_{11}^{-1}(t) B^{(1),*}(t) \Big\} r^1(t)  \nonumber\\
&- \Big( \int_{t}^T  \Phi^*(s,t)  H(s)  \Psi_{K^1}(s, t) ds\Big) \Big(B^{(2)}(t) 
-B^{(1)}(t)R_{11}^{-1}(t) R_{12}(t) \Big)
\overline{u^{2,o}}(t), \nonumber \\
&-  \Phi^*(T,t)M(T) \Psi_{K^1}(T, t) \Big( B^{(2)}(t) -B^{(1)}(t)R_{11}^{-1}(t) R_{12}(t) \Big)  \overline{u^{2,o}}(t)   
   \hso t \in [0,T), \hso r^1(T)=0,   \label{ex49g} \\
\dot{r}^2(t)=& \Big\{  -A^{*}(t) +\Phi^*(T,t) M(T)  \Psi_{K^2}(T,t) B^{(2)}(t) R_{22}^{-1}(t)B^{(2),*}(t) \nonumber \\
&+ \Big(\int_{t}^T \Phi^*(s,t) H(s)  \Psi_{K^2}(s,t) ds \Big) B^{(2)}(t) R_{22}^{-1}(t) B^{(2),*}(t) \Big\}r^2(t) \nonumber \\
&- \Big( \int_{t}^T  \Phi^*(s,t)  H(s)  \Psi_{K^2}(s, t) ds\Big) \Big(B^{(1)}(t)  -B^{(2)}(t)R_{22}^{-1}(t) R_{21}(t) \Big)    \overline{u^{1,o}}(t)  \nonumber \\
&-  \Phi^*(T,t)M(T) \Psi_{K^2}(T, t) \Big(B^{(1)}(t)  -B^{(2)}(t)R_{22}^{-1}(t) R_{21}(t) \Big)    \overline{u^{1,o}}(t),
  \hso t \in [0,T), \hso r^2(T)=0,    \label{ex49h}
\end{align}  
\begin{align}
\dot{\overline{x^o}}(t)=A(t) \overline{x^o}(t) + B^{(1)}(t) \overline{u^{1,o}}(t)+ B^{(2)}(t) \overline{u^{2,o}}(t), \hst \overline{x^o}(0)=\overline{x}_0, \label{ex101}
\end{align}
\bea
   \left[ \begin{array}{c} \overline{u^{1,o}}(t) \\ \overline{u^{2,o}}(t) \end{array} \right] = -    \left[ \begin{array}{cc} I & R_{11}^{-1}(t)R_{12}(t)  \\ 
R_{22}^{-1}(t)R_{21}(t) & I \end{array} \right]^{-1}      \left[ \begin{array}{c} R_{11}^{-1}(t) B^{(1),*}(t) \Big\{K^1(t) \overline{x^o}(t) +r^1(t) \Big\} \\
R_{22}^{-1}(t) B^{(2),*}(t) \Big\{ K^2(t)    \overline{x^o}(t) +r^2(t)\Big\} \end{array} \right]. \label{ex106}
\eea

\end{theorem}
\begin{proof}  
Since $u_t^{1,o}$ is ${\cal G}_{0,t}^{W^1}-$measurable and $u_t^{2,o}$ is ${\cal G}_{0,t}^{W^2}-$measurable, and ${\cal G}_{0,t}^{W^1}$ and ${\cal G}_{0,t}^{W^2}$ are independent,  then 
\begin{align}
\pi^{w^1}(u^2)(s,t) &={\mathbb E} \Big( u_s^2 | {\cal  G}_{0,t}^{W^1} \Big) ={\mathbb E} \Big( u_s^2\Big) \equiv \overline{u^2}(s), \hso t \leq s \leq T, \label{in2} \\ \pi^{w^2}(u^1)(s,t) &={\mathbb E} \Big( u_s^1 | {\cal  G}_{0,t}^{W^2} \Big) ={\mathbb E} \Big( u_s^1 \Big) \equiv \overline{u^1}(s), \hso t \leq s \leq T. \label{in2a}
\end{align}
Substituting (\ref{in2}), (\ref{in2a}) into (\ref{pr1}), (\ref{pr2}), and then  (\ref{pr1}), (\ref{pr2}) into (\ref{ex49}) we have 
\begin{align}
\pi^{w^1}(\psi^{o})(t)=& \Big\{ \Phi^*(T,t)M(T) \Psi_{K^1}(T,t) + \int_{t}^T \Phi^*(s,t)  H(s) \Psi_{K^1}(s,t)ds \Big\} \pi^{w^1}(x^o)(t) \nonumber \\
&+ \Phi^*(T,t)M(T) \int_{t}^T \Psi_{K^1}(T, \tau)\Big( B^{(2)}(\tau)-B^{(1)}(\tau)R_{11}^{-1}(\tau) R_{12}(\tau) \Big) \overline{u^{2,o}}(\tau) d\tau \nonumber \\ 
&+  \int_{t}^T \Phi^*(s,t)  H(s)  \int_{t}^s \Psi_{K^1}(s, \tau)\Big( B^{(2)}(\tau)  -B^{(1)}(\tau)R_{11}^{-1}(\tau) R_{12}(\tau) \Big) \overline{u^{2,o}}(\tau) d\tau ds \nonumber \\
&-\Phi^*(T,t) M(T) \int_{t}^T \Psi_{K^1}(T,\tau) B^{(1)}(\tau) R_{11}^{-1}(\tau)B^{(1),*}(\tau) r^1(\tau) d\tau \nonumber \\
&-\int_{t}^T \Phi^*(s,t) H(s) \int_{t}^s \Psi_{K^1}(s,\tau) B^{(1)}(\tau) R_{11}^{-1}(\tau) B^{(1),*}(\tau)r^1(\tau) d \tau ds,
 \label{ex49a}
\end{align}  
\begin{align}
\pi^{w^2}(\psi^{o})(t)=& \Big\{ \Phi^*(T,t)M(T) \Psi_{K^2}(T,t) + \int_{t}^T \Phi^*(s,t)  H(s) \Psi_{K^2}(s,t)ds \Big\} \pi^{w^2}(x^o)(t) \nonumber \\
&+ \Phi^*(T,t)M(T) \int_{t}^T \Psi_{K^2}(T, \tau) \Big( B^{(1)}(\tau)  
-B^{(2)}(\tau)R_{22}^{-1}(\tau) R_{21}(\tau) \Big) 
 \overline{u^{1,o}}(\tau) d\tau \nonumber \\ 
&+  \int_{t}^T \Phi^*(s,t)  H(s)  \int_{t}^s \Psi_{K^2}(s, \tau) \Big(B^{(1)}(\tau)   -B^{(2)}(\tau)R_{22}^{-1}(\tau) R_{21}(\tau) \Big)              \overline{u^{1,o}}(\tau) d\tau ds\nonumber \\
&-\Phi^*(T,t) M(T) \int_{t}^T \Psi_{K^2}(T,\tau) B^{(2)}(\tau) R_{22}^{-1}(\tau)B^{(2),*}(\tau) r^2(\tau) d\tau \nonumber \\
&-\int_{t}^T \Phi^*(s,t) H(s) \int_{t}^s \Psi_{K^2}(s,\tau) B^{(2)}(\tau) R_{22}^{-1}(\tau) B^{(2),*}(\tau)r^2(\tau) d \tau ds.
 \label{ex49b}
\end{align}  
Comparing (\ref{in54}) with the previous two equations then $K^i(\cdot), i=1,2$ are identified by the operators
\begin{align}
K^i(t)=  \Phi^*(T,t)M(T) \Psi_{K^i}(T,t) + \int_{t}^T \Phi^*(s,t)  H(s) \Psi_{K^i}(s,t)ds, \hst t \in [0,T], \hso i=1,2, \label{ex49c}
\end{align}
and $r^i(\cdot), i=1,2$ by the processes 
\begin{align}
r^1(t)=&\Phi^*(T,t)M(T) \int_{t}^T \Psi_{K^1}(T, \tau)\Big( B^{(2)}(\tau)-B^{(1)}(\tau)R_{11}^{-1}(\tau) R_{12}(\tau) \Big) \overline{u^{2,o}}(\tau) d\tau \nonumber \\ 
&+  \int_{t}^T \Phi^*(s,t)  H(s)  \int_{t}^s \Psi_{K^1}(s, \tau)\Big( B^{(2)}(\tau)  -B^{(1)}(\tau)R_{11}^{-1}(\tau) R_{12}(\tau) \Big) \overline{u^{2,o}}(\tau) d\tau ds \nonumber \\
&-\Phi^*(T,t) M(T) \int_{t}^T \Psi_{K^1}(T,\tau) B^{(1)}(\tau) R_{11}^{-1}(\tau)B^{(1),*}(\tau) r^1(\tau) d\tau \nonumber \\
&-\int_{t}^T \Phi^*(s,t) H(s) \int_{t}^s \Psi_{K^1}(s,\tau) B^{(1)}(\tau) R_{11}^{-1}(\tau) B^{(1),*}(\tau)r^1(\tau) d \tau ds,    \label{ex49d} 
\end{align}
\begin{align}
r^2(t)=& \Phi^*(T,t)M(T) \int_{t}^T \Psi_{K^2}(T, \tau) \Big( B^{(1)}(\tau)  
-B^{(2)}(\tau)R_{22}^{-1}(\tau) R_{21}(\tau) \Big) 
 \overline{u^{1,o}}(\tau) d\tau \nonumber \\ 
&+  \int_{t}^T \Phi^*(s,t)  H(s)  \int_{t}^s \Psi_{K^2}(s, \tau) \Big(B^{(1)}(\tau)   -B^{(2)}(\tau)R_{22}^{-1}(\tau) R_{21}(\tau) \Big)              \overline{u^{1,o}}(\tau) d\tau ds\nonumber \\
&-\Phi^*(T,t) M(T) \int_{t}^T \Psi_{K^2}(T,\tau) B^{(2)}(\tau) R_{22}^{-1}(\tau)B^{(2),*}(\tau) r^2(\tau) d\tau \nonumber \\
&-\int_{t}^T \Phi^*(s,t) H(s) \int_{t}^s \Psi_{K^2}(s,\tau) B^{(2)}(\tau) R_{22}^{-1}(\tau) B^{(2),*}(\tau)r^2(\tau) d \tau ds.  \label{ex49e}
\end{align}  
Differentiating both sides of (\ref{ex49c}) the operators $K^i(\cdot), i=1,2$ satisfy the following matrix differential equations (\ref{ex20pi}), (\ref{ex49f}). 
Differentiating both sides of (\ref{ex49d}), (\ref{ex49e})  the processes $r^i(\cdot), i=1,2$ satisfy the   differential equations (\ref{ex49g}), (\ref{ex49h}). Utilizing (\ref{in2}), (\ref{in2a}) we obtain the optimal strategies (\ref{ex49i}), (\ref{ex49j}).
Next, we determine $\overline{u^{i,o}}$ for $i=1,2$ from (\ref{ex49i}), (\ref{ex49j}). \\
Define the averages 
\begin{align}
\overline{x}(t) \tri {\mathbb E}\Big\{x(t)\Big\}= {\mathbb E} \Big\{  \pi^{w^i}(x)(t)  \Big\}, \hst i=1, 2.
\end{align}
Then $\overline{x^o}(\cdot)$ satisfies the ordinary differential equation (\ref{ex101}).
Taking the expectation of both sides of (\ref{ex49i}), (\ref{ex49j}) we deduce the corresponding  equations
\begin{align}
\overline{u^{1,o}}(t) &=  -R_{11}^{-1}(t) B^{(1),*}(t) \Big\{K^1(t) \overline{x^o}(t) +r^1(t) \Big\}  -R_{11}^{-1}(t)R_{12}(t) \overline{u^{2,o}}(t),   \hst t \in [0,T],    \label{ex102} \\
\overline{u^{2,o}}(t)  &=  -R_{22}^{-1}(t) B^{(2),*}(t) \Big\{ K^2(t)    \overline{x^o}(t) +r^2(t)\Big\} -R_{22}^{-1}(t)R_{21}(t) \overline{u^{1,o}}(t), \hst t \in [0,T].  \label{ex103}
\end{align}
The last two equations can be written in matrix form (\ref{ex106}). This completes the derivation.
\end{proof}

\noi Hence, the optimal strategies are computed from (\ref{ex49i}), (\ref{ex49j}), where the filter equations for $\pi^{w^i}(x^o)(\cdot), i=1,2$ satisfy (\ref{ex500}), (\ref{ex510}), while  $\Big(K^i(\cdot), r^i(\cdot), \overline{u^{i,o}}(\cdot), \overline{x^o}(\cdot) \Big),i=1,2$ are computed off-line utilizing  the ordinary differential equations (\ref{ex20pi}), (\ref{ex49f}), (\ref{ex49g}), (\ref{ex49h}), (\ref{ex101}), (\ref{ex106}). Note that the optimal decentralized strategy $u^{1,o}$ given by (\ref{ex49i})  is a linear function of the state $x^{1,o}(\cdot)$ and ${\mathbb E}( u^{2,o})(\cdot)$, while the state is governed by  (\ref{ex500}) corresponding to $u^{2,o}$ replaced by its average value ${\mathbb E}( u^{2,o})(\cdot)$, and similarly for $u^{2,o}$. The optimal strategies can be further simplified  by considering special structures of  interconnected dynamics, such as, coupling of the subsystems via the DM's, coupling through the pay-off only,  diagonal matrices $R =diag\{R_{11}, R_{12}\}, H=diag\{H_{11}, H_{22}\}, M= diag\{M_{11}, M_{22}\}$, etc.

\noi Further,  Theorem~\ref{dise} can be generalized to an arbitrary number of interconnected system team games. In addition, one may consider feedback information structures, delayed information structures, etc.. \\  These generalization or simplification are stated in the next remark.

\begin{remark}
\label{remPIP}
(Generalizations and Simplifications){\ \\}
\noi{\bf Generalizations.} Theorem~\ref{dise} is easily generalized to the following arbitrary coupled  dynamics
\begin{align}
dx^i(t) =&A_{ii}(t)x^i(t) dt +  B^{(i)} u_t^i dt+ G_{ii} dW^i(t)  \nonumber \\
&+ \sum_{j=1, j\neq i}^N A_{ij}x^j(t) dt + \sum_{ j=1, j\neq i }^N B^{(j)}(t)u_t^j dt , \hso x^i(0)=x_0^i, \hso t \in (0,T], \hso i \in {\mathbb Z}_N \label{R1}
\end{align}
and DM's information structures 
\bea
u_t^i \hso \mbox{is} \hso {\cal G}_{0,t}^{W^i}-\mbox{measurable}, \hso t \in [0,T], \hso i \in {\mathbb Z}_N. \label{R4}
\eea 
The optimal strategies are obvious extensions of the ones given in Theorem~\ref{dise}.\\

\noi{\bf Simplifications.} Several simpler forms can be deduced from the results of Theorem~\ref{dise} by assuming any of the following $R =diag\{R_{11}, R_{12}\}, H=diag\{H_{11}, H_{22}\}, M= diag\{M_{11}, M_{22}\}$. Moreover, simplified strategies can be derived by assuming nested information structures, that is, $u_t^1$ is ${\cal G}_{0,t}^{W^1}-$measurable and $u_t^2$ is ${\cal G}_{0,t}^{W^1,W^2}-$measurable. \\

\noi{\bf Delay Information Structures.} The optimality conditions  hold for any ${\cal G}_{0,t}^i-$measurable DM strategies $u^i, i =1, \ldots, N$. Therefore, one can apply the necessary conditions to DM's information structures 
\bea
u_t^i \hso \mbox{is} \hso {\cal G}_{0,t-\eps_i}^{W^i}-\mbox{measurable}, \hso \eps_i >0, \hso  t \in [0,T], \hso i \in {\mathbb Z}_N. \label{R5}
\eea 
or any other information structures of interest, such as, delayed sharing. \\

\noi{\bf Feedback Information Structures.} The previous generalizations/simplifications also apply to  feedback information structures ${\cal G}_{0,t}^{z^{i,u}}$. Specifically, to derive the corresponding results of Theorem~\ref{dise}, even for the simplest scenario $z^1=x^1, z^2=x^2$, one has to compute  conditional expectations with respect to ${\cal G}_{0,t}^{x^{i,u}}, i=1, 2$, and hence one has to invoke  nonlinear filtering techniques to determine expressions for the filters  $\pi^{x^i}(x)(t) \tri {\mathbb E} \Big\{ x(t) | {\cal G}_{0,t}^{x^{i,u} }\Big\}, i=1, 2$, $\pi^{x^2}(u^1)(t) \tri {\mathbb E} \Big\{ u_t^1 | {\cal G}_{0,t}^{x^{2,u} }\Big\}, \pi^{x^1}(u^2)(t) \tri {\mathbb E} \Big\{ u_t^2 | {\cal G}_{0,t}^{x^{1,u} }\Big\},$ (and predictions of  $x(t), u_t^1, u_t^2$).  It appears to us that the optimal team laws are the same as those derived for nonanticipative information structures, given by (\ref{ex49i}), (\ref{ex49j}), with $\pi^{w^i}(x)(t)$ replaced by $\pi^{x^i}(x)(t), i=1, 2$ and $\overline{u^{1,o}}(t), \overline{u^{2,o}}(t)$ replaced by  $\pi^{x^2}(u^1)(t), \pi^{x^1}(u^2)(t)$. These estimates (filters)  may not be described in terms of linear Kalman-type equations driven by the DMs strategies governing the  conditional means, whose gains are specified by  the conditional error covariance equations,  independently of the observations. A possible approach is to  compute these conditional expectations is the identification of  a sufficient statistic as in \cite{charalambous1997,charalambous-elliott1997,charalambous-elliott1998}. \\

\noi{\bf Signaling.} Given the optimal decentralized strategies of Theorem~\ref{dise} we can determine the amount of signaling among the DMs to reduce the computational complexity of the optimal strategies.

\end{remark}

\section{Conclusions and Future Work}
\label{cf}
In this second part of our two-part paper, we invoke the stochastic maximum principle, conditional Hamiltonian and the coupled backward-forward stochastic differential equations of the first part \cite{charalambous-ahmedFIS_Parti2012} to derive team optimal decentralized strategies for   distributed stochastic differential systems with noiseless information structures. We present examples of such team games of nonlinear  as well as linear quadratic forms. In some cases we obtain closed form expressions of the optimal  decentralized strategies.  

\noi The methodology is very general, and applicable to several types of information structures such as the ones described under Remark~\ref{remPIP}. It will be interesting to consider additional types of information structures and compute the optimal decentralized strategies in closed form, to better understand the implications of signaling and computational complexity of such strategies compared to centralized strategies.

\bibliographystyle{IEEEtran}
\bibliography{bibdata}

\end{document}